\documentclass[1p,preprint,10pt]{elsarticle}
\usepackage[usenames,dvipsnames,svgnames,table]{xcolor}

\makeatletter
\def\ps@pprintTitle{%
 \let\@oddhead\@empty
 \let\@evenhead\@empty
 \def\@oddfoot{}%
 \let\@evenfoot\@oddfoot}
\newcommand{\mvec}{\mathop{\operator@font vec}}
\makeatother

\usepackage{amsfonts}
\usepackage{amsmath}
\usepackage{amsthm}
\usepackage{amssymb}
\usepackage{bbm}

\usepackage{mathtools}
\usepackage{etextools}
\usepackage{ifthen}

\usepackage{mathrsfs}
\usepackage{mathdots}
\usepackage{algorithmic}

\usepackage{pgfplotstable}
\usepackage{tikz}
\usepackage[backref=false, colorlinks, breaklinks, hyperfootnotes, plainpages=false, pdfpagelabels=false]{hyperref}

\usepackage{cleveref}

\renewcommand{\vec}[1]{\ensuremath{\boldsymbol{#1}}} % bold vectors
\newcommand{\bmx}{\begin{bmatrix}}
\newcommand{\emx}{\end{bmatrix}}
\newcommand{\bsm}{\left[\begin{smallmatrix}}
\newcommand{\esm}{\end{smallmatrix}\right]}

\makeatletter
\newcommand{\rrank}{\mathop{\operator@font rank}}
\newcommand{\diag}{\mathop{\operator@font diag}}
\newcommand{\tvec}{\mathop{\operator@font vec}}
\makeatother

\usepackage{stmaryrd}
\newcommand{\MylinkColor}{blue}	%we change the color later in black of course
\hypersetup{linkcolor=\MylinkColor, menucolor=\MylinkColor, urlcolor=\MylinkColor, citecolor=\MylinkColor, anchorcolor=\MylinkColor}
\def\scrA{\mathscr{A}}

\def\calG{\mathcal{G}}

\def\calJ{\mathcal{J}}

\def\calL{\mathcal{L}}

\def\calQ{\mathcal{Q}}

\def\calT{\mathcal{T}}

\def\calX{\mathcal{X}}

\def\bfa{\mathbf{a}}

\def\bfc{\mathbf{c}}

\def\bff{\mathbf{f}}
\def\bfg{\mathbf{g}}
\def\bfh{\mathbf{h}}

\def\bfu{\mathbf{u}}
\def\bfv{\mathbf{v}}
\def\bfw{\mathbf{w}}
\def\bfx{\mathbf{x}}
\def\bfy{\mathbf{y}}
\def\bfz{\mathbf{z}}

\def\bfA{\mathbf{A}}

\def\bfD{\mathbf{D}}

\def\bfH{\mathbf{H}}

\def\bfJ{\mathbf{J}}

\def\bfM{\mathbf{M}}

\def\bfU{\mathbf{U}}
\def\bfV{\mathbf{V}}
\def\bfW{\mathbf{W}}

\def\bfZ{\mathbf{Z}}

\def\bbR{\mathbb{R}}

\newtheorem{theorem}{Theorem} 
 
\newtheorem{corollary}{Corollary} 
\newtheorem{lemma}{Lemma} 
\newtheorem{proposition}{Proposition} 
 
\newtheorem{remark}{Remark} 
\newtheorem{example}{Example}

\def\TDIS{\calJ}
\def\TVVU{\calQ}

\def\LDIS{\calL_\calJ}
\def\LVVU{\calL_\calQ}

\def\VDM{\bfA}

\newcommand\HomT[1]{\Psi^{(#1)}}
\newcommand\mmult[1]{\bullet_{#1}}

\begin{document}
\begin{frontmatter}
\author[c]{Konstantin Usevich\corref{cor}}\ead{konstantin.usevich@univ-lorraine.fr}
\author[b]{Philippe Dreesen}\ead{philippe.dreesen@vub.ac.be}
\author[b]{Mariya Ishteva}\ead{mariya.ishteva@vub.ac.be}
\cortext[cor]{Corresponding author}
\address[c]{Universit\'e de Lorraine, CNRS, CRAN, F-54000 Nancy, France}
\address[b]{Vrije Universiteit Brussel (VUB), Department VUB-ELEC, Brussels, Belgium}
\title{Decoupling multivariate polynomials:\\interconnections between tensorizations}

\begin{abstract}
Decoupling multivariate polynomials is useful for obtaining an insight into the workings of a nonlinear mapping, performing parameter reduction, or approximating nonlinear functions. 
Several different tensor-based approaches have been proposed independently for this task, involving different tensor representations of the functions, and ultimately leading to a canonical polyadic decomposition.

We first show that the involved tensors are related by a linear transformation, and that their CP decompositions and uniqueness properties are closely related.
This connection provides a way to better assess which of the methods should be favored in certain problem settings, and may be a starting point to unify the two approaches. Second, we show that taking into account the previously ignored intrinsic structure in the tensor decompositions improves the uniqueness properties of the decompositions and thus enlarges the applicability range of the methods.
\end{abstract}

\end{frontmatter}

%\linenumbers

\section{Introduction}
Representing a nonlinear function in a simpler way can provide an insight into its inner workings, reduce the parametric complexity, or facilitate function approximation. 
{One of the successful examples are tensor decompositions, such as the canonical polyadic (CP) decomposition that can be viewed as a decomposition of functions into a sum of separable functions \cite{Comon14ISPM-Tensors}.
Tensor decompositions found many applications in signal/image processing, chemometrics, physics, machine learning, to name a few \cite{kolda2009tdaa,Cichocki15:review,Ng2017igarss,sidiropoulos2017tensor,Acar.etal15PI-Data}.
In these applications, tensors either appear naturally due to multi-dimensionality of data \cite{Chen2015TV}, or the data can be tensorized, i.e., a higher-order tensor is constructed from data \cite{debals2015}.
  }

{ In this paper, we focus on }
%For this reason,
the task of decoupling a set of polynomial vector functions, that is, decomposing a set of multivariate real polynomials into linear combinations of univariate polynomials in linear forms of the input variables.
{This task has attracted a spark of research attention over the last years, motivated by several applications, such} as  system identification \cite{dreesen2014i2mtc,dreesen2016isma,tiels2013fctdpripwm,Schoukens.Rolain12TIM-Cross,Schoukens.etal14conf-Identification,fakhrizadeh2018mssp}, approximation theory \cite{Logan.Shepp75D-Optimal,Lin.Pinkus93JoAT-Fundamentality,Oskolkov02SMJ-Representations}, { and  neural networks} \cite{Shin.Ghosh96IToNN-Ridge}. 
%It has also been applied successfully in the field of 
%}
Restricting { polynomial decoupling} to a single homogeneous polynomial is equivalent to the well-known Waring decomposition \cite{iarrobino-kanev1999,landsberg2012tensorsgeometry}, but { some} generalizations to non-homogeneous polynomials or joint Waring decompositions are studied as well \cite{Bialynicki-Birula.Schinzel08CM-Representations,schinzel2002} and \cite{carlini2003waringseveralforms,Dreesen.etal14-Decoupling,tiels2013fctdpripwm}.
%The same type of decomposition appears in other fields, such as approximation theory \cite{Logan.Shepp75D-Optimal,Lin.Pinkus93JoAT-Fundamentality,Oskolkov02SMJ-Representations} and polynomial neural networks \cite{Shin.Ghosh96IToNN-Ridge}. 
%It has also been applied successfully in the field of system identification \cite{dreesen2014i2mtc,dreesen2016isma,tiels2013fctdpripwm,Schoukens.Rolain12TIM-Cross,Schoukens.etal14conf-Identification,fakhrizadeh2018mssp}. 

{ Several tensor-based approaches were proposed for computing a decoupled representation of a given function }
%The polynomial decoupling problem has led to several tensor-based solution methods 
\cite{Schoukens.Rolain12TIM-Cross,Schoukens.etal14conf-Identification,Tiels.Schoukens14conf-coupled,VanMulders.etal14conf-Identification,Dreesen.etal14-Decoupling}. 
These { solutions} can be categorized into two classes. 
The methods \cite{Schoukens.Rolain12TIM-Cross,Schoukens.etal14conf-Identification,Tiels.Schoukens14conf-coupled,VanMulders.etal14conf-Identification} build a tensor from the polynomial coefficients, whereas the method of \cite{Dreesen.etal14-Decoupling} builds a tensor from the Jacobian matrices of the functions, evaluated at a set of sampling points. 
Ultimately, all methods boil down to a canonical polyadic decomposition (CP decomposition) of the constructed tensor to retrieve a decoupled representation in which the nonlinearities occur as univariate polynomial mappings.  

The benefit of using a tensor-based approach for decoupling is twofold. 
First, `tensorization' procedures often lead to (essentially) uniquely decomposable tensors \cite{debals2015}, i.e., ensuring that 
identifiable structures can be retrieved. 
%the parameters of identifiable problems can be retrieved.
Second, by solving the decoupling problem as a CP decomposition, one can use recent widely available and robust numerical tools, such as Tensorlab for MATLAB \cite{tensorlab3} (or alternatives \cite{andersson2000,bader2012tensortoolbox}). 

This paper specifically focuses on the two tensorization methods \cite{VanMulders.etal14conf-Identification} and \cite{Dreesen.etal14-Decoupling}.
Although both associated tensors have a particular structure, both approaches seem quite different in nature, and each of the methods has distinct advantages over the other one. 
For instance, the coefficient-based methods \cite{Schoukens.Rolain12TIM-Cross,Schoukens.etal14conf-Identification,Tiels.Schoukens14conf-coupled,VanMulders.etal14conf-Identification} require several high-order tensors (or their matricizations) for polynomials of high degrees, whereas \cite{Dreesen.etal14-Decoupling} involves a single third-order tensor only. 
Coefficient-based approaches can easily deal with single polynomials, whereas \cite{Dreesen.etal14-Decoupling} would in that case not be able to take advantage of the uniqueness properties of the CP decomposition, as the \emph{tensor} of Jacobian matrices is then a matrix composed of gradient vectors. 
On the other hand, the approach of \cite{Dreesen.etal14-Decoupling} can be applied to non-polynomial functions, which may in some cases be of interest, e.g., in \cite{dreesen2014i2mtc} a neural network was decoupled. 

We aim at obtaining a deeper understanding of the connections between the solution approaches. This is profitable when extending the applicability range of the methods, e.g., when moving from polynomials to any differentiable functions. %of interest in a non-algebraic or non-polynomial context.  
Furthermore, such connections may provide a way to transfer theoretical properties from one formulation to another.  For example, as we argue in Section~\ref{sec:structured}, exploring the previously ignored structure in the tensor decomposition in one of the settings enlarges the range of decomposable functions. This knowledge may lead to improved algorithms in another setting as well.
%We will explore how the different decoupling approaches can be related and we discuss their properties.

The remainder of this article is organized as follows: 
Section~\ref{sec:model} formalizes the problem of decoupling multivariate polynomials.
Section~\ref{sec:waring} explains the link between the decoupling problem and the symmetric tensor decomposition problem.
Section~\ref{sec:tensorizations} discusses the construction of the tensor of unfoldings~\cite{VanMulders.etal14conf-Identification} and the Jacobian tensor~\cite{Dreesen.etal14-Decoupling}.
%Section~\ref{sec:relations} presents the relations between the two alternatives and associated computational aspects. 
%Section~\ref{sec:quasi-Hankel} recalls links between polynomials and tensors.
Section~\ref{sec:relations} presents our first contribution, namely the relation between the two tensorizations.
The second main contribution of the paper is Section~\ref{sec:structured}, which clarifies the need of dealing with structure in the decompositions and proposes a coupled CP decomposition approach for solving the structured problem.
Section~\ref{sec:conclusions} draws the conclusions and points out open problems for future work. 

\subsection*{Notation}
Scalars are denoted by lowercase or uppercase letters. 
Vectors are denoted by lowercase boldface letters, e.g., $\bfu$. 
Elements of a vector are denoted by lowercase letters with an index as subscript, e.g., $\bfx = \left[ \begin{array}{ccc} x_1 & \cdots & x_m \end{array} \right]^\top$. 
Matrices are denoted by uppercase boldface letters, e.g., $\bfV$. 
The entry in the $i$-th row and $j$-th column of a matrix $\bfV$ is denoted by $v_{ij}$, and the matrix $\bfV \in \bbR^{m \times r}$ may be represented by its columns $\bfV = \left[ \begin{array}{ccc} \bfv_1 & \cdots & \bfv_r \end{array} \right]$.
The Kronecker product of matrices is denoted by ``$\otimes$''.

Tensors of order $d$ are denoted by uppercase caligraphical letters, e.g., $\calJ \in \bbR^{n \times m \times N}$.
The outer product is denoted by ``$\circ$'' and is defined as follows: For $\calT = \bfu \circ \bfv \circ \bfw$, the entry in position $(i,j,k)$ is equal to $u_i v_j w_k$.
The canonical polyadic (CP) decomposition expresses a tensor $\calT$ as a (minimal) sum of rank-one tensor terms \cite{carroll1970,harshman1970,kolda2009tdaa} as 
$\calT = \sum_{i=1}^R \bfu_i \circ \bfv_i \circ \bfw_i$,
%see Figure~\ref{fig:structured_CPD},
and is sometimes denoted in a short-hand notation as 
$\calT = \llbracket \bfU, \bfV, \bfW \rrbracket$.
The CP rank $r$ is defined as the (minimal) number of terms that is required to represent $\calT$ as a sum of $r$ rank-one terms. 
To refer to elements of matrices or tensors, or subsets thereof, we may use MATLAB-like index notation (including MATLAB's colon wildcard): for instance,  $\calT_{i,j,k,\ell}$ is the element at position $(i,j,k,\ell)$ of a fourth-order tensor $\calT$, and $\calT_{:,:,2}$ is the second frontal slice of a third-order tensor $\calT$. 
The mode-$n$ product is denoted by ``$\bullet_n$'' and is defined as follows. 
Let $\calX$ be an $I_1 \times I_2 \times \cdots \times I_N$ tensor,
and let $\vec{u}$ be a vector of length $I_n$,
then we have 
$\left( \calX \bullet_n \vec{u}^\top \right)_{i_1 \cdots i_{n-1} i_{n+1} \cdots i_N } = \sum_{i_n=1}^{I_n} x_{i_1 i_2 \cdots i_N} u_{i_n}$.
Notice that the result is a tensor of order $N-1$, as mode $n$ is summed out.
Similarly, for an $I_1 \times I_2 \times \cdots \times I_N$ tensor  $\calX$ and a matrix $\bfM \in J \times I_n$, the mode-$n$ product is defined as $\left( \calX \bullet_n \bfM^\top \right)_{i_1 \cdots i_{n-1} j i_{n+1} \cdots i_N } = \sum_{i_n=1}^{I_n} x_{i_1 i_2 \cdots i_N} m_{j,i_n}$.
Let $\tvec(\calT)$ denote the  column-major vectorization of a tensor $\calT$.
The first-mode unfolding of an $I_1 \times I_2 \times \cdots \times I_N$ tensor $\calX$ is the matrix $\calX_{(1)}$ of size $I_1 \times I_2 \cdots I_N$, where each row is the vectorized slice of the tensor $\calX$, i.e. $(\calX_{(1)})_{i,:} = \left(\tvec(\calX_{i,:,\ldots,:})\right)^{\top}$
(see, for example, \cite{kolda2009tdaa} for more details).

%\begin{figure}[htb!]
%\centering
%\begin{tikzpicture}
%\begin{scope}[scale=0.75]
%%\draw(0.5,0.5) node (A) {\small $\mathcal{T}(\bff)$};
%\draw(0.5,0.5) node (A) {};
%\draw(0,0) rectangle (1,1);
%\draw(0,1) -- (0.5,1.5) -- (1.5,1.5) -- (1.5,0.5) --(1,0);
%\draw(1,1) -- (1.5,1.5);
%
%\draw(2,0.8) node (eq) {$=$};
%\draw(6,0.8) node (eq) {$\quad+\,\cdots\,+\quad$};
%
%\draw(3,-0.2) -- (3,0.8);
%\draw(3.2,1) -- (4.2,1);
%\draw(3.2,1.2) -- (3.7,1.7);
%
%\draw(8,-0.2) -- (8,0.8);
%\draw(8.2,1) -- (9.2,1);
%\draw(8.2,1.2) -- (8.7,1.7);
%\end{scope}
%\end{tikzpicture}
%\caption{The CPD decomposes a third-order tensor into a minimal sum of rank-one terms, where each rank-one term is the outer product of three vectors. 
%The number of terms is called the rank of the tensor. 
%    }
%\label{fig:structured_CPD}
%\end{figure}

\section{The polynomial decoupling model}\label{sec:model}
%\subsection{Problem statement}
First, we describe the model, following the notation of \cite{Dreesen.etal14-Decoupling} as illustrated in Fig.~\ref{fig:decoupled}.
Consider a multivariate polynomial map $\bff: \bbR^{m} \to \bbR^{n}$, i.e., a vector 
\[
\bff(\bfu) = \bmx f_1(\bfu) & \cdots & f_{n}(\bfu) \emx^{\top}
\]
of multivariate polynomials (of total degree at most $d$) in variables $\bfu = \bmx u_1 & \cdots & u_{m} \emx^{\top}$. 
We say that  $\bff$ has a \textit{decoupled representation}, if it can be expressed as
\begin{equation}\label{eq:decoupled}
\bff(\bfu) = \bfW \bfg(\bfV^{\top}\bfu), 
%\bff(\vec{u}) = \bfW \cdot \bfg(\bfV^{\top}\bfu), 
\end{equation}
where $\bfV \in \bbR^{m \times r}, \bfW \in \bbR^{n \times r}$ are transformation matrices, and $\bfg: \bbR^{r} \to \bbR^{r}$ is defined as 
\[
\bfg(x_1, \ldots, x_{r}) = \bmx g_1(x_1) & \cdots& g_{r}(x_{r}) \emx^{\top},
\]
where $g_k: \bbR \to \bbR$  are univariate polynomials of degree at most $d$, i.e.,
\begin{equation}\label{eq:g_k}
g_k(t) = c_1 t + \cdots + c_d t^d.
\end{equation}
Note that we omitted the constant terms of the polynomials, since they are not uniquely identifiable \cite{Dreesen.etal14-Decoupling}. 
In this paper we limit ourselves to the model \eqref{eq:decoupled}.
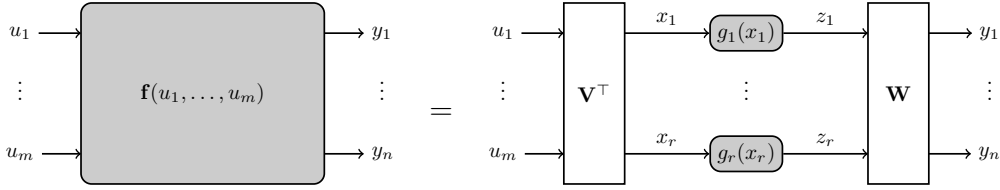
\begin{figure}[htb!]
    \centering
    \tikz \node [scale=0.8]{
        \begin{tikzpicture}
            \node (u1) at (0,3) {$u_1$};
            \node at (0,2.15) {$\vdots$};
            \node (um) at (0,1) {$u_m$};
            \draw [thick,fill=black!20, rounded corners=5pt] (1,0.5) rectangle (5,3.5); \node (F) at (3,2) {$\bff(u_1,\ldots,u_m)$};
            \draw [->, thick, label=] (5,3) -- (5.65,3) node[right] {$y_1$};
            \node at (5.95,2.15) {$\vdots$};
            \draw [->, thick] (5,1) -- (5.65,1) node[right] {$y_n$};
            \draw [->, thick] (u1) -- (1,3);
            \draw [->, thick] (um) -- (1,1);
        \end{tikzpicture}
        \quad
        \raisebox{6\height}{\Large$=$}
        \quad
        \begin{tikzpicture}
            \node (u1) at (0,3) {$u_1$};
            \node at (0,2.15) {$\vdots$};
            \node (um) at (0,1) {$u_m$};
            \draw [thick] (1,0.5) rectangle (2,3.5); \node (L) at (1.5,2) {$\bfV^{\top}$};
            \draw [->, thick] (u1) -- (1,3);
            \draw [->, thick] (um) -- (1,1);
            \node [shape=rectangle,draw,thick,fill=black!20,rounded corners=5pt] (g1) at (4,3) {$g_1(x_1)$};
            \draw [->, thick, label=] (2,3) -- (g1) node[above,midway] {$x_1$};
            \node at (4,2.15) {$\vdots$};
            \node [shape=rectangle,draw,thick,fill=black!20,rounded corners=5pt] (gr) at (4,1) {$g_r(x_r)$};
            \draw [->, thick] (2,1) -- (gr) node[above,midway] {$x_r$};
            \draw [thick] (6,0.5) rectangle (7,3.5); \node (R) at (6.5,2) {$\bfW$};
            \draw [->, thick] (g1) -- (6,3) node[above,midway] {$z_1$};
            \draw [->, thick] (gr) -- (6,1) node[above,midway] {$z_r$};
            \node (y1) at (8,3) {$y_1$};
            \node at (8,2.15) {$\vdots$};
            \node (yn) at (8,1) {$y_n$};
            \draw [->, thick] (7,3) -- (y1);
            \draw [->, thick] (7,1) -- (yn);
        \end{tikzpicture} 
    };
    \caption{Every multivariate polynomial vector function $\bff(\bfu)$ can be represented by a linear transformation of a set of univariate functions (in linear combinations of the original variables).}\label{fig:decoupled}
		 \end{figure}
		 
The decoupled representation \eqref{eq:decoupled} can be also equivalently rewritten as
\begin{equation}\label{eq:decoupling_additive}
\bff(\bfu) = \bfw_1 g_1(\bfv^{\top}_1\bfu) + \cdots + \bfw_r g_r(\bfv^{\top}_r\bfu), 
\end{equation}
where $\bfv_k$ and $\bfw_k$ are the columns of $\bfV$ and $\bfW$, respectively. %, {i.e.}, $\bfV = \bmx \bfv_1 & \cdots & \bfv_r \emx$, and $\bfW = \bmx \bfw_1 & \cdots & \bfw_r \emx$.
As shown in \cite{Comon.etal15conf-polynomial,comon2017xrank}, the decomposition \eqref{eq:decoupling_additive} is a special case of the $X$-rank decomposition \cite[\S 5.2.1]{Landsberg12-Tensors}, where the set of ``rank-one'' terms is the set of polynomial maps of the form $\bfw g(\bfv^{\top}\bfu)$.
The $X$-rank framework is useful \cite{comon2017xrank} for studying the identifiability of the model \eqref{eq:decoupling_additive}.

The following example shows a decoupled representation for a simple case.
This example will be used throughout the paper to illustrate the main ideas of the various aspects that we will explore. 
\begin{example}\label{ex:main_example}
    Consider a function $\bff(\bfu) = \left[ \begin{array}{cc} f_1(u_1,u_2) & f_2(u_1,u_2) \end{array} \right]^\top$ given as 
\[
\begin{split}
    f_1(u_1,u_2) &= - 3 u_1^3 - 9 u_1^2 u_2 - 27 u_1 u_2^2 - 15 u_2^3 - 8 u_1^2 - 8 u_1 u_2 - 20 u_2^2 + 3 u_1 + 9 u_2,\\
    f_2(u_1,u_2) &= - 7 u_1^3 - 6 u_1^2 u_2 + 6 u_1 u_2^2 + 7 u_2^3 + 10 u_1^2 + 16 u_1 u_2 + 10 u_2^2 - 3 u_2. 
\end{split}
\]
It can be verified that $\bff$ has a decomposition \eqref{eq:decoupling_additive} with $m=n=2$ and $r=3$ as  
\[
    \bfV = \left[ \begin{array}{rrr} 2 & -1 & 1 \\ 1 & 1 & 2 \end{array} \right], 
        \quad \mbox{and} \quad 
    \bfW = \left[ \begin{array}{rrr} 0 & 1 & -2 \\ -1 & 0 & 1 \end{array} \right],  
\]
and $g_1(x_1) = x_1^3 -2 x_1^2 - x_1$, $g_2(x_2) = x_2^3 -4 x_2^2 + x_2$, $g_3(x_3)= x_3^3 +2 x_3^2 -2 x_3$ (see Figure~\ref{fig:f1f2}).  
%\end{split}
%\]
{\color{blue}
\begin{figure}[h!]
\includegraphics[width=0.48\textwidth]{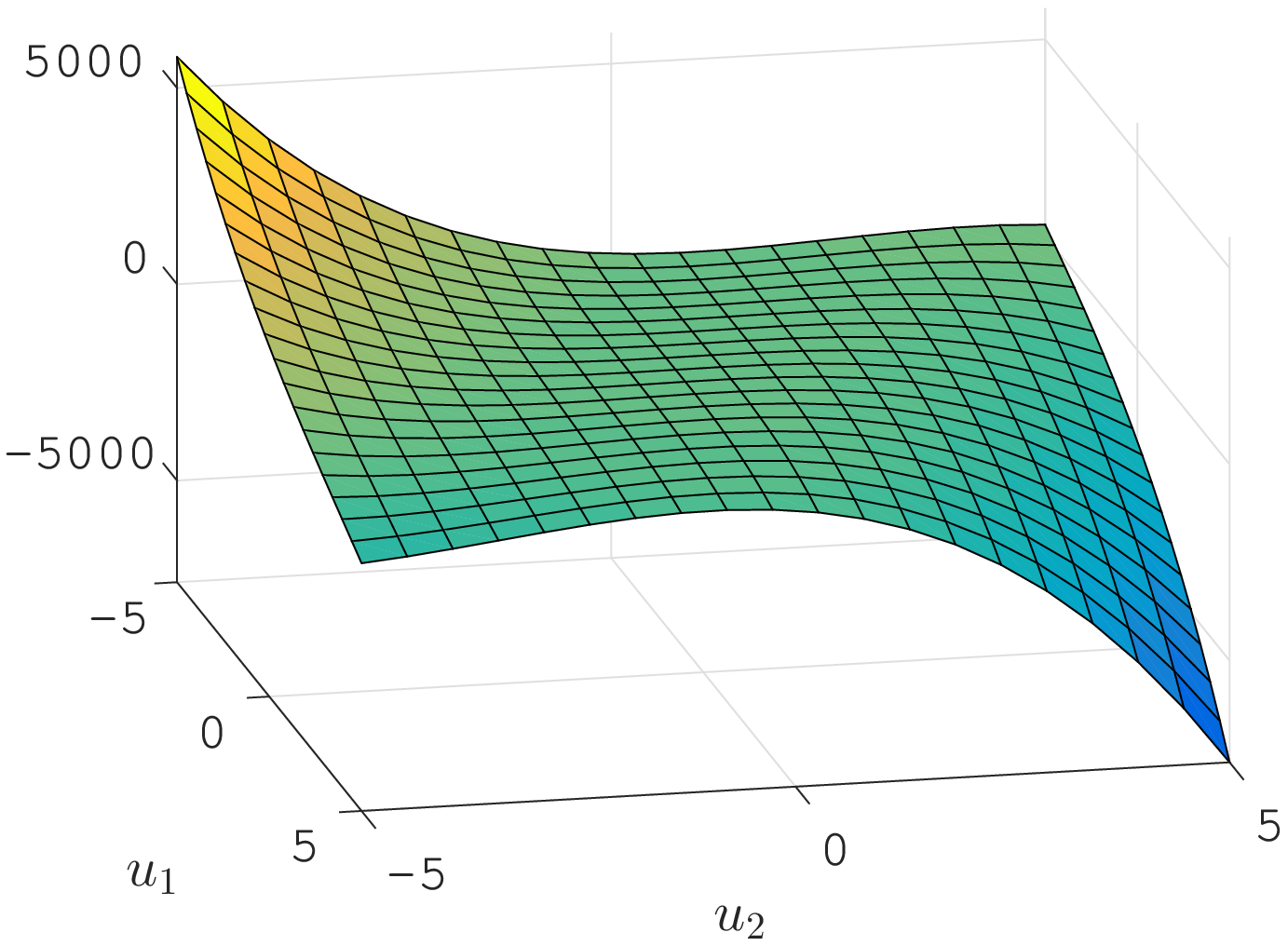}
\quad
\includegraphics[width=0.48\textwidth]{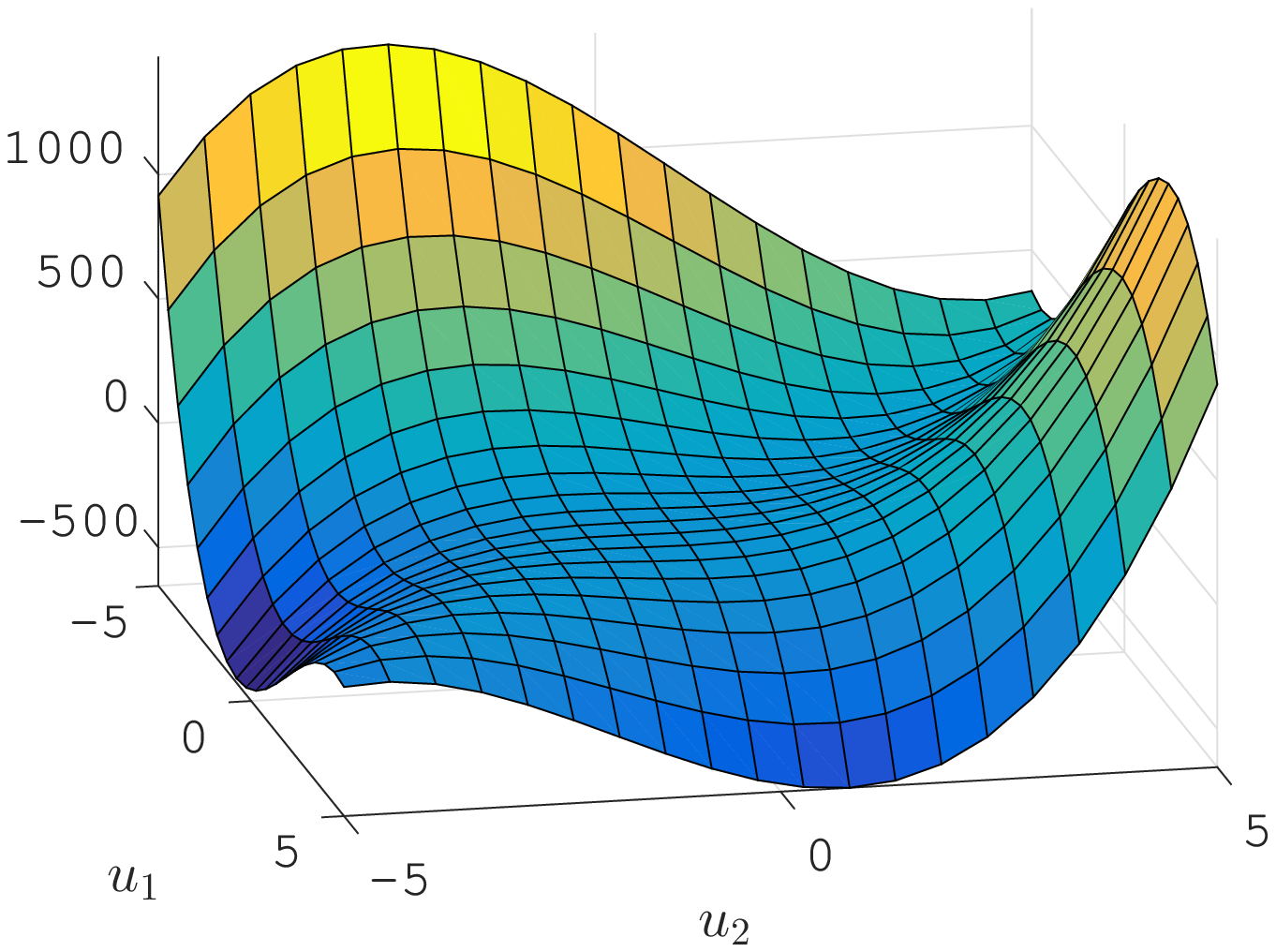}
\\[2em]
\includegraphics[width=0.31\textwidth]{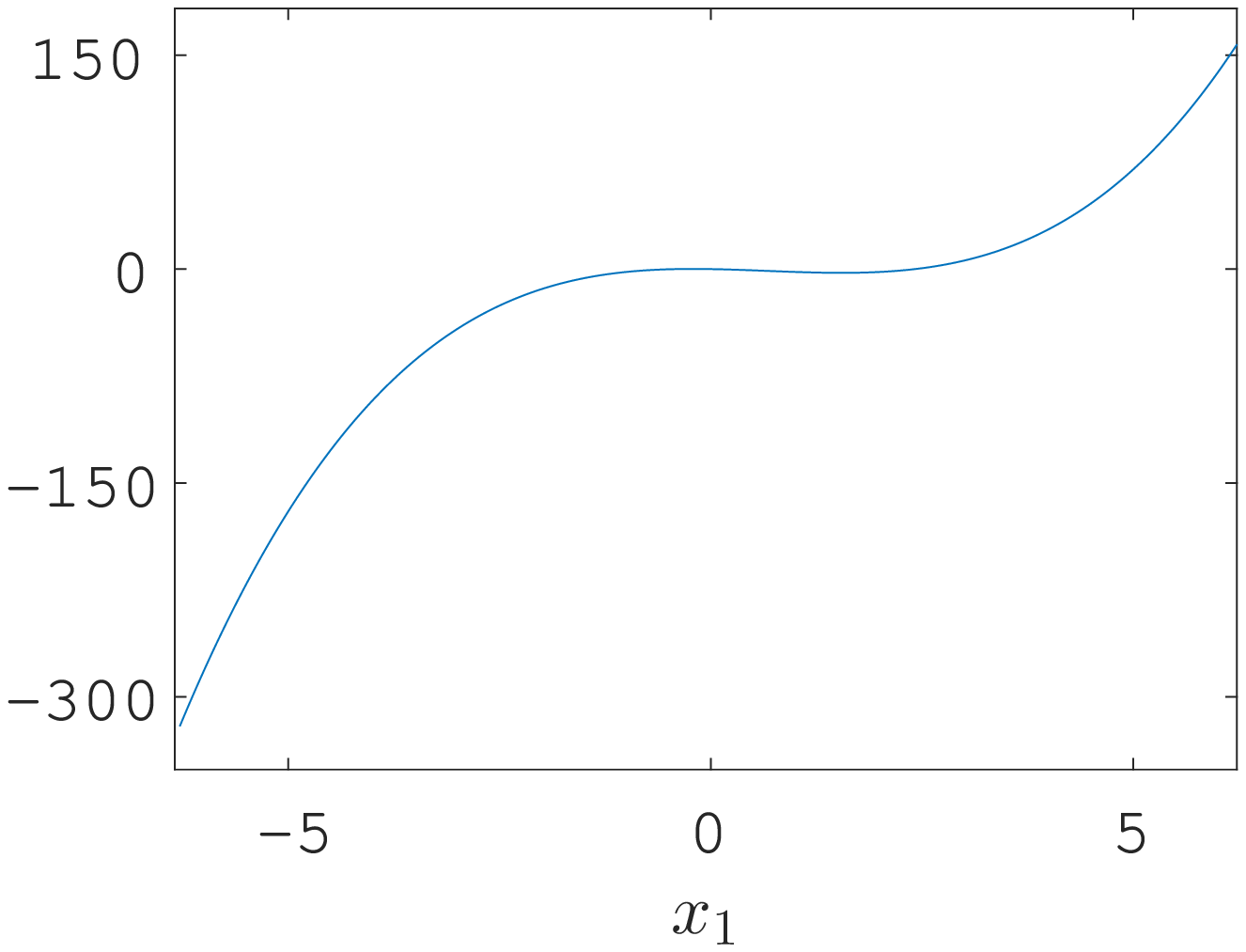}
\quad
\includegraphics[width=0.31\textwidth]{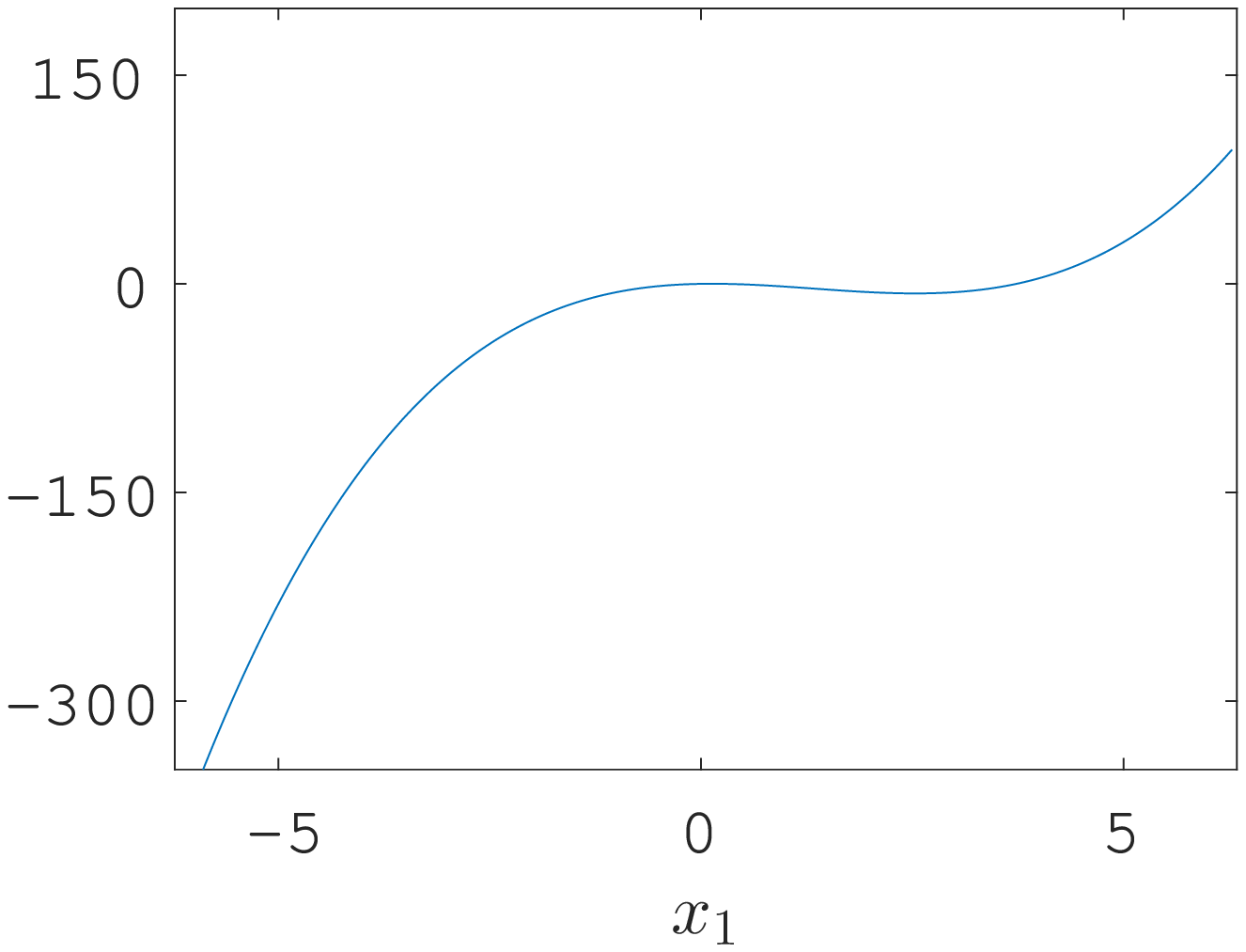}
\quad
\includegraphics[width=0.31\textwidth]{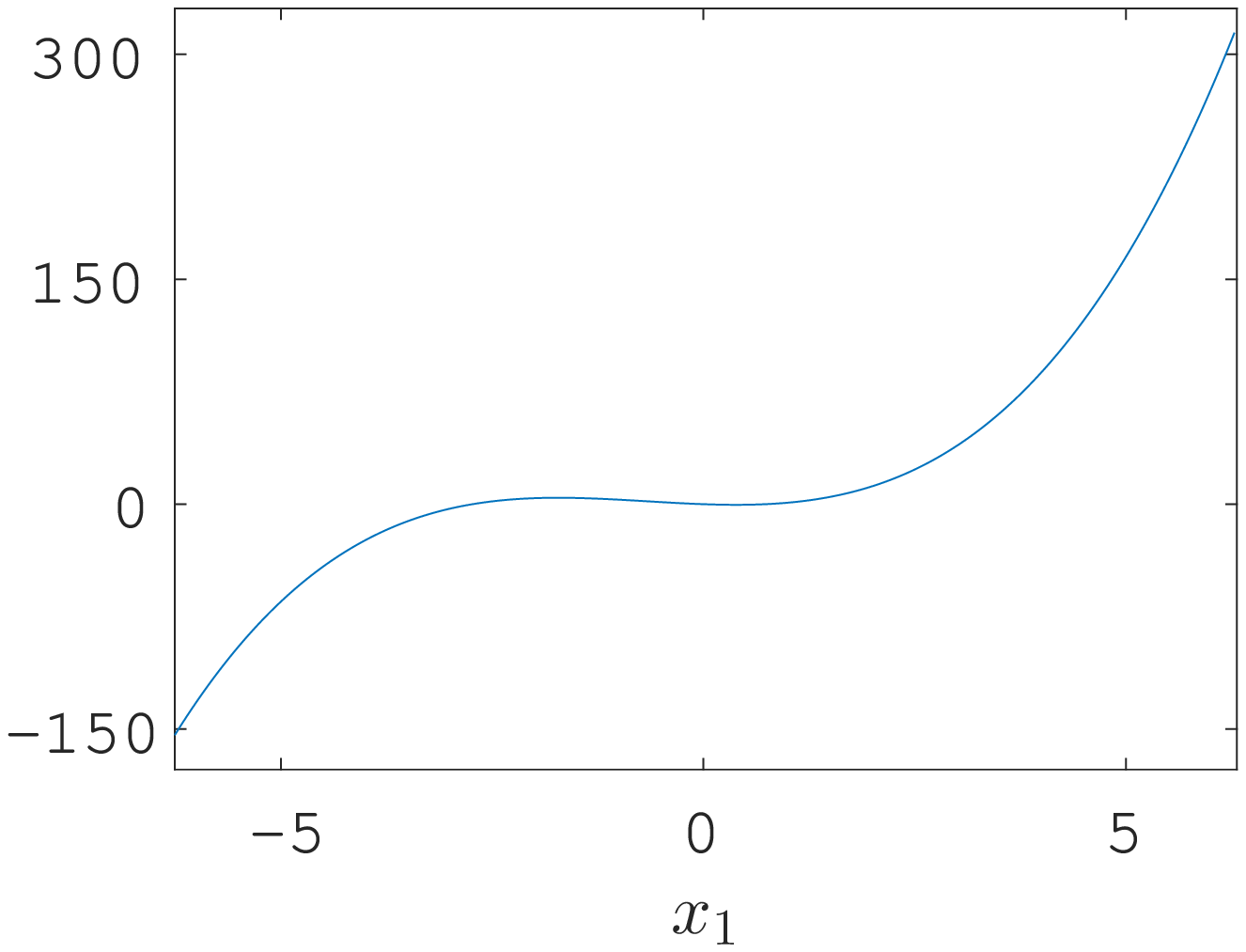}
    \caption{The decoupling problem in Example 1 consists of decomposing the multivariate functions $f_1(u_1,u_2)$ and $f_2(u_1,u_2)$ (top row) to the univariate functions $g_1(x_1)$, $g_2(x_2)$, and $g_3(x_3)$ (bottom row), using suitable transformation matrices as in (\ref{eq:decoupled}).}
\label{fig:f1f2}
\end{figure}
}
\end{example}

{\color{blue}
\begin{remark}
In general, the coupled representation $\bff(\bfu)$ has 
%\[
    $    n \left( {m+d \choose d} -1\right)$
%\]
coefficients, while the decoupled representation $\bfW \bfg (\bfV^\top \bfu)$ has 
%\[
$r(m + n + d)$
%\]
coefficients. 
Due to the combinatorial increase of the number of coefficients in the coupled representation, the decoupled representation is especially beneficial for large values of $m$, $n$, and $d$. 
But even for small values of $m$, $n$, and $d$, the parametric reduction can be significant, for example, if $m=n=3$, $d=5$, and $r=3$, the coupled representation has 168 coefficients, while the decoupled one has only 36 coefficients. 
\end{remark}
}

%\pdlong{\color{red}
%\subsection{Decoupling as (structured) canonical polyadic decomposition}
%The main idea of the different tensorization approaches is to construct from a polynomial a tensor $\calT(\bff)$ such that 
%\begin{center}
%\fbox{the polynomial $\bff$ has a decoupled representation \eqref{eq:decoupled} if and only if $\calT(\bff)$  has CP rank $r$.}
%\end{center}
%The CP rank $r$ is defined as the (minimal) number of terms that is required to represent $\calT$ as a sum of $r$ rank-one terms. 
%
%Probably the most successful tensorization is the one of \cite{Dreesen.etal14-Decoupling}. 
%But other tensorizations were also used: in \cite{VanMulders.etal14conf-Identification},  \cite{Schoukens.etal14conf-Identification}, \cite{Tiels.Schoukens14conf-coupled}, \cite{Schoukens.Rolain12TIM-Cross}.
%
%The first goal of this paper is to describe the connection between the tensorizations of \cite{Dreesen.etal14-Decoupling} and \cite{VanMulders.etal14conf-Identification}.
%
%The second goal is to explore the structure of the tensor. Indeed, both tensors in 
%\cite{Dreesen.etal14-Decoupling} and \cite{VanMulders.etal14conf-Identification} are structured, and it is important to take the structure into account because of the following:
%
%\begin{itemize}
%\item An unstructured CPD may be non-unique, but the structured (see Figure~\ref{fig:structured_CPD} can be still unique).
%\item Approximate decoupling (in presence of noise) can be reformulated as a structure-preserving approximation of the tensor $\calT(\bff)$.
%\end{itemize}
%
%}

\section{Decoupling polynomials and symmetric tensor decompositions}\label{sec:waring}
Let us review some well-known facts that connect polynomials with symmetric tensors  \cite{Comon.etal08SJMAA-Symmetric,batselier2016}, and that connect some special cases of the representation \eqref{eq:decoupled} with symmetric tensor decompositions.

\subsection{Homogeneous polynomials, symmetric tensors and Waring decomposition}
It is well-known that there is a one-to-one correspondence between homogeneous polynomials and symmetric tensors \cite{Comon.etal08SJMAA-Symmetric}.
For instance, the polynomial $-8 u_1^2  -8 u_1 u_2 - 20 u_2^2$ can be written as \begin{equation}\label{eq:quadratic_example}
-8 u_1^2 -8 u_1 u_2 - 20 u_2^2 
= \bfu^\top \HomT{2} \bfu,\quad \mbox{where }\HomT{2} = \left[ \begin{array}{rr} -8 & -4 \\ -4 & -20 \end{array} \right].
\end{equation}
In general, let $p^{(d)}(u_1,\ldots, u_m)$ be a homogeneous polynomial (also called a $d$-ary form) of degree $d$ in $m$ variables. 
Then there is a unique symmetric tensor $\Psi^{(d)}$ of order $d$ and dimension $m$ such that
\begin{equation}\label{eq:HomPolyT}
p(\bfu) = \HomT{d} \mmult{1} \bfu \cdots \mmult{d} \bfu.
\end{equation}
Next, it is easy to see that the decoupling problem for the polynomial \eqref{eq:HomPolyT} 
takes the form
\begin{equation}\label{eq:waring}
    p(u_1,\ldots,u_m) = \sum_{i=1}^r w_i (v_{1i} u_1 + \cdots + v_{mi} u_m)^d, 
\end{equation}  
which is known as the Waring decomposition \cite{iarrobino-kanev1999, landsberg2012tensorsgeometry} of $p(u_1,\ldots,u_m)$.
The Waring decomposition, in its turn, is equivalent to the symmetric CP decomposition of $\HomT{d}$:
\[
\HomT{d} = \sum_{i=1}^r w_i (\bfv_i \circ \cdots \circ \bfv_i).
\] 
\color{black}
The symmetric CP decomposition of $\HomT{d}$ reveals possible values for the unknowns $v_{ij}$ and $w_i$.
\begin{example}\label{ex:quadratic}
Consider the polynomial given in \eqref{eq:quadratic_example}. % second degree part of $f_1(u_1,u_2)$ from Example~\ref{ex:symmtensors}.
Then the corresponding symmetric matrix $\HomT{2}$ admits the decomposition
\begin{equation}
\left[ \begin{array}{rr} -8 & -4 \\ -4 & -20 \end{array} \right]
= 
\left[ \begin{array}{rr} -2 & 2 \\ 2 & 4 \end{array} \right] 
\left[ \begin{array}{rr} -1 & 0 \\ 0 & -1 \end{array} \right] 
\left[ \begin{array}{rr} -2 & 2 \\ 2 & 4 \end{array} \right], 
\end{equation}
    such that $p(u_1,u_2) = \bfu^\top \HomT{2} \bfu$ has the Waring decomposition 
\[ p(u_1,u_2) = - (- 2 u_1 + 2 u_2)^2 - (2 u_1 + 4 u_2)^2. \]
\end{example}
Notice that the symmetric decomposition of $\HomT{2}$ from \Cref{ex:quadratic} is not unique (nor `essentially unique' \cite{kolda2009tdaa}). Indeed, the eigenvalue decomposition
\[
\left[ \begin{array}{rr} -8 & -4 \\ -4 & -20 \end{array} \right]
\approx
\left[\begin{array}{rr}0.2898 & -0.9571 \\ 0.9571 & 0.2898\end{array}\right] 
\left[\begin{array}{rr} -21.2111 & 0 \\ 0 &   -6.7889 \end{array}\right] 
\left[\begin{array}{rr} 0.2898 & 0.9571 \\  -0.9571 &   0.2898 \end{array} \right] 
\]
provides another valid factorization. 
For $d>2$, however, the Waring decomposition \eqref{eq:waring} possesses uniqueness properties even in the case of quite large ranks \cite{Domanov13uniqueness,Chiantini.etal17generic}. 

Along the same lines, it is possible to decouple jointly several homogeneous polynomials. 
Consider the case of $n$ homogeneous polynomials of degree $d$, denoted by
\begin{equation}\label{eq:decoupling_sim_HomT}
    \begin{array}{rcl}
        p_1(u_1,\ldots,u_m) &=& \HomT{d}_1 \mmult{1} \bfu \cdots \mmult{d} \bfu,\\ 
        &\vdots& \\
        p_n(u_1,\ldots,u_m) &=& \HomT{d}_n \mmult{1} \bfu \cdots \mmult{d} \bfu.\\ 
    \end{array}
\end{equation}
Then the decoupling problem \eqref{eq:decoupled} corresponds to the simultaneous Waring decomposition of several forms or, equivalently, the coupled CP decomposition of several symmetric  tensors.
The rank and identifiability properties of simultaneous Waring decompositions were also studied in the literature, see \cite{carlini2003waringseveralforms,Domanov13uniqueness,Abo18Most} and references therein. 

\subsection{The case of non-homogeneous polynomials}
Next, consider the case of non-homogeneous polynomials.
Any non-homogeneous polynomial of degree $d$ can hence be written as
\begin{equation}\label{eq:poly_symt_form}
p(\bfu) = \bfu^\top \HomT{1}  + \bfu^{\top} \HomT{2} \bfu  + 
\HomT{3} \mmult{1} \bfu \mmult{2} \bfu \mmult{3} \bfu + \cdots
+\HomT{d} \mmult{1} \bfu \cdots \mmult{d} \bfu,
\end{equation}
where $\HomT{1} \in \bbR^{m}$ , $\HomT{2} \in \bbR^{m\times m}$ is a symmetric matrix, and each $\HomT{s} \in \bbR^{m\times \cdots \times m}$, $3 \le s \le d$, is a symmetric tensor of order $s$.
\begin{example}\label{ex:symmtensors}
We continue Example~\ref{ex:main_example}. 
We can write $f_1(u_1,u_2)$ and $f_2(u_1,u_2)$ as
\[
f_1(u_1,u_2) 
%&=& 
=
\bfu^\top \left[\begin{array}{r}
3\\9
\end{array}
\right] 
+ 
\bfu^\top
\left[\begin{array}{rr}
    -8 & -4 \\-4 & -20
\end{array} \right]
\bfu
+ 
\HomT{3} 
\mmult{1} \bfu \mmult{2} \bfu \mmult{3} \bfu,
\]
with 
\[
\HomT{3}_{:,:,1} = \left[\begin{array}{rr}
 -3 &  -3 \\
 -3 &  -9  
\end{array}\right],
\quad \quad \quad
\HomT{3}_{:,:,2} = \left[\begin{array}{rr}
 -3  & -9 \\
 -9 &  -15 
\end{array}\right],
\]
and
\[
f_2(u_1,u_2) = 
\bfu^\top \left[\begin{array}{r}
0 \\ -3
\end{array}
\right] 
+ 
\bfu^\top
\left[\begin{array}{rr}
 10 & 8 \\ 8 & 10
\end{array} \right]
\bfu
+ 
\HomT{3} 
\mmult{1} \bfu \mmult{2} \bfu \mmult{3} \bfu,
\]
with 
\[
\HomT{3}_{:,:,1} = \left[ \begin{array}{rr}
 -7 &  -2\\
 -2 &   2 
\end{array}\right],
\quad \quad \quad
\HomT{3}_{:,:,2} = \left[ \begin{array}{rr}
 -2  & 2 \\
 2 &  7 
\end{array}\right].
\]
\end{example}
The decomposition of a single non-homogeneous polynomial as in
\eqref{eq:decoupling_additive} is hence equivalent to joint decomposition of several symmetric tensors but of different orders \cite{Comon.etal15conf-polynomial}.

Finally, several non-homogeneous polynomials can be jointly decomposed in a similar way. 
Consider $n$ non-homogeneous polynomials of maximal degree $d$, denoted as
\begin{equation}
    \begin{array}{rcl}
        p_1(u_1,\ldots,u_m) &=& \bfu^\top \HomT{1}_1 + \cdots + \HomT{d}_1 \mmult{1} \bfu \cdots \mmult{d} \bfu, \\ 
        &\vdots& \\
        p_n(u_1,\ldots,u_m) &=& \bfu^\top \HomT{1}_n + \cdots + \HomT{d}_n \mmult{1} \bfu \cdots \mmult{d} \bfu,  
    \end{array}
\end{equation}
The full decomposition in \eqref{eq:decoupled} can be also viewed as a coupled tensor decomposition, which will be presented in  \Cref{sec:simcoupsymCPD}.

\section{Tensorizations and their decompositions}
\label{sec:tensorizations}
In this section, we recall tensorizations proposed in the literature to find the decomposition \eqref{eq:decoupled} by a CP decomposition of a single tensor constructed from $\bff$, namely the tensorizations of  \cite{VanMulders.etal14conf-Identification} and \cite{Dreesen.etal14-Decoupling}.
We recall basic properties and give short proofs for completeness, although these proofs are already present in \cite{VanMulders.etal14conf-Identification, Dreesen.etal14-Decoupling}.
We also use a slightly different notation to simplify the exposition.

\subsection{Tensor of unfoldings \cite{VanMulders.etal14conf-Identification}}
The above link between polynomials, (partially) symmetric tensors and their CP decompositions gives rise to the tensorization approach of \cite{VanMulders.etal14conf-Identification}, in which a tensor is constructed from the coefficients of the polynomials $f_1(u_1,\ldots,u_m)$ up to $f_n(u_1,\ldots,u_m)$.  
This tensorization offers the advantage that several polynomials can be represented as a single tensor, and the decoupling task can be solved using a single (but structured) CP decomposition. 
In this approach, the tensor (shown in Figure~\ref{fig:tensor_VVU}) is constructed from the coefficients of the polynomial map of degree $d$, as follows:
\begin{itemize}
\item The tensor has size $n \times m \times \delta$, where $\delta = \sum\limits_{k=1}^{d} m^{k-1}$.
\item The tensor is constructed by slices
\[
\TVVU_{i,:,:} := \Psi(f_{i}),
\]
where $\Psi$ is a structured $m \times \delta$ matrix built from the coefficients of $f_{i}(\bfu)$.
\end{itemize}

\begin{figure}[htb!]
\centering
\begin{tikzpicture}
\begin{scope}[scale=0.75]
\draw(-6,-1) rectangle (-4,1);
\draw(-6,1) -- (-4.5,2.5) -- (-2.5,2.5) -- (-2.5,0.5) --(-4,-1);
\draw(-4,1) -- (-2.5,2.5);

\draw(-1.5,0) node (D) {\small $=$};
\draw(-5,0) node (A) {\small $\TVVU$};

\draw(0.5,0) node (D) {\small $\vdots$};

\draw[fill=white](0,-1.25) -- (1.5,0.25) -- (3.5,0.25) -- (2,-1.25) -- (0,-1.25);
\draw[fill=white](0,0.5) -- (1.5,2) -- (3.5,2) -- (2,0.5) -- (0,0.5);
\draw[fill=white](0,1.25) -- (1.5,2.75) -- (3.5,2.75) -- (2,1.25) -- (0,1.25);

\draw(1.25,1.5) node (B) {\small $\Psi(f_1)$};
\draw(1.25,0.75) node (B) {\small $\Psi(f_2)$};
\draw(1.25,-1) node (B) {\small $\Psi(f_n)$};

\draw(-5,-1.25) node (D) {\small $m$};
\draw(-6.25,0) node (D) {\small $n$};
\draw(-5.5,2) node (D) {\small $\delta$};

\end{scope}
\end{tikzpicture}
    \caption{The coefficients of a polynomial map $\bff : \bbR^m \to \bbR^n$ of degree $d$ can be arranged into an $n \times m \times \delta$ tensor $\calQ$, where $\delta = \sum_{k=1}^d m^{k-1}$.}\label{fig:tensor_VVU}
\end{figure}
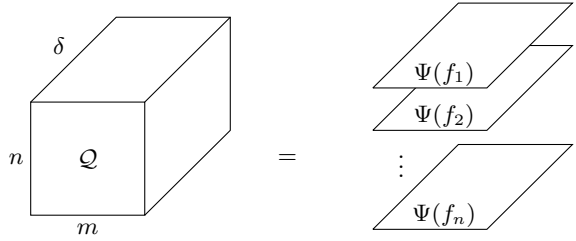
Now let us describe the construction of the structured coefficient matrix $\Psi(p)$ for a given polynomial of degree $d$.
Recall that each such polynomial can be written as in \eqref{eq:poly_symt_form},
where $\HomT{1} \in \bbR^{m}$ , $\HomT{2} \in \bbR^{m\times m}$  is a symmetric matrix and $\HomT{s} \in \bbR^{m\times \cdots \times m}$ are symmetric tensors of order $s$.
Then the matrix $\Psi(p) \in \bbR^{m \times \delta}$ is constructed\footnote{In  \cite{VanMulders.etal14conf-Identification} the linear term is skipped, and $\delta = \sum\limits_{k=2}^{d} m^{d-1}$. In  \cite{VanMulders.etal13A-Identification} the  matrix $\Psi$ is denoted as $\Gamma$.} as
\begin{equation}\label{eq:def_Psi}
\Psi(p) = \left[\begin{array}{c|c|c|c|c} \HomT{1} & \HomT{2} & \HomT{3}_{(1)} & \cdots & \HomT{d}_{(1)} \end{array}\right],
\end{equation}
where $\calG_{(1)}$ denotes the first-mode unfolding of a tensor $\calG$.

\begin{example}\label{ex:main_example_TVVU}
A third-degree polynomial in two variables
\[
p(u_1,u_2) =  a_1 u_1 + a_2 u_2  + b_1 u_1^2 + 2 b_2 u_1u_2 + b3 u_2^2 + d_1 u_1^3 +3d_2 u_1^2 u_2 + 3d_3 u_1 u_2^2 + d_4 u_3^3
 \]
 has the representation
\begin{equation}\label{eq:poly2_example}
p(u_1,u_2) = \bfu^{\top} \bmx a_{1} \\ a_{2}\emx + 
\bfu^{\top} \bmx b_1 & b_2 \\ b_2 & b_3 \emx \bfu + \HomT{3} \mmult{1} \bfu \mmult{2} \bfu \mmult{3} \bfu,
\end{equation}
where 
\[
\HomT{3}_{:,:,1} = \bmx d_1 & d_2 \\ d_2 & d_3 \emx, \quad \HomT{3}_{:,:,2} = \bmx d_2 & d_3 \\ d_3 & d_4 \emx.
\]
By putting all the unfoldings together, we get  
\begin{equation}\label{eq:Psi_p_poly2}
\Psi(p) = 
\left[
\begin{array}{c|cc|cccc}
a_{1} & b_1 & b_2   & d_1  & d_2   & d_2  &  d_3  \\
a_{2} & b_2 & b_3  &  d_2 & d_3 & d_3 & d_4
\end{array}
\right].
\end{equation}
Hence, for  $f_1$ and $f_2$ in Example~\ref{ex:main_example}, the slices of the tensor $\TVVU$ are given by
\[
\TVVU_{1,:,:} =\Psi(f_1)  = \left[\begin{array}{c|cc|cccc}
3 & -8 & -4   &  -3  &  -3 & -3  & -9 \\
9 & -4 & -20  &  -3 &  -9  & -9 &  -15 
\end{array}\right],
\]
and 
\[
\TVVU_{2,:,:} =\Psi(f_2)=\left[\begin{array}{c|cc|cccc}
0 & 10 &  8  &  -7 &  -2 & -2  & 2 \\
-3 & 8 & 10  &  -2 &   2  & 2 &  7 
\end{array}\right].
\]
\end{example}

%\begin{example}\label{ex:poly2_slice_unfolding}
%We  with an example $m=2$ and $d=3$. Take a polynomial
%\begin{equation}\label{eq:poly2_example}
%p(u_1,u_2) = \bfu^{\top} \bmx\psi^{(1)}_{1} \\ \psi^{(1)}_{2}\emx + 
%\bfu^{\top} \bmx\psi^{(2)}_{1,1} & \psi^{(2)}_{1,2} \\ \psi^{(2)}_{1,2} & \psi^{(1)}_{2,2} \emx \bfu + \HomT{3} \mmult{1} \bfu \mmult{2} \bfu \mmult{3} \bfu.
%\end{equation}
%Then the matrix $\Psi(p) \in \bbR^{2\times (1+2+4)}$ has the following form
%\begin{equation}\label{eq:Psi_p_poly2}
%\Psi(p) = 
%\left[
%\begin{array}{c|cc|cccc}
%\psi^{(1)}_{1} & \psi^{(2)}_{1,1} & \psi^{(2)}_{1,2}     &  {\psi^{(3)}}_{1,1,1}  &  {\psi^{(3)}}_{2,1,1}   &  {\psi^{(3)}}_{2,1,1}  &  {\psi^{(3)}}_{2,2,1}  \\
%{\psi^{(1)}}_{2} & \psi^{(2)}_{1,2}   & \psi^{(2)}_{2,2}  &  {\psi^{(3)}}_{2,1,1}  &  {\psi^{(3)}}_{2,2,1}   & {\psi^{(3)}}_{2,2,1} &  {\psi^{(3)}}_{2,2,2} 
%\end{array},
%\right]
%\end{equation}
%where $\psi^{(3)}_{i,j,k}$ is the $(i,j,k)$-th element of the tensor $\HomT{3}$.
%Note that in \eqref{eq:Psi_p_poly2}, some elements are equal due to symmetry.
%\end{example}
As proved in \cite{VanMulders.etal14conf-Identification}, the tensor $\TVVU$ has a CP decomposition, which reveals the decomposition (\ref{eq:decoupled}).
We repeat here a simplified version of the proof for completeness.
\begin{lemma}
\label{lem:VVU_CPD}
\label{lem:TVVU_decoupled_cp}
For the polynomial map \eqref{eq:decoupled}, the tensor $\TVVU$ has the following CP decomposition:
\begin{equation}\label{eq:TVVU_CP}
\TVVU = \sum\limits_{k=1}^{r} \bfw_k \circ \bfv_k \circ \bfz_k,
\end{equation}
where 
\begin{equation}\label{eq:z_k}
\bfz_k = \left[\begin{array}{c|c|c|c|c} c_{k,1} & c_{k,2} \bfv_{k}^{\top} & c_{k,3}(\bfv_{k}\otimes \bfv_{k})^{\top} & \cdots & c_{k,d} (\bfv_{k}\otimes \cdots \otimes \bfv_{k})^{\top} \end{array}\right]^{\top}.
\end{equation}
\end{lemma}
\begin{proof}
Consider  $q_k(\bfu) := g_k(\bfv_k^{\top} \bfu)$, where $g_k$ is as in \eqref{eq:g_k}. Easy calculations  show that
\[
\Psi(q_k) =  \bfv_k \bfz^{\top}_k,
\]
see also \cite[eqn. (A.7)]{VanMulders.etal13A-Identification}.
Since, from \eqref{eq:decoupling_additive}  $f_{i}(\bfu) = \sum\limits_{k=1}^r (\bfw_k)_{i} q_k(\bfu)$, we have that
\[
\Psi(f_i) =  \sum\limits_{k=1}^r (\bfw_k)_{i} \bfv_k \bfz^{\top}_k
\]
which implies \eqref{eq:TVVU_CP}.
\end{proof}

\begin{example}
We continue Examples~\ref{ex:main_example},~\ref{ex:main_example_TVVU}. 
The Kronecker products of columns of $\bfV$ are:
\[
(\bfv_1 \otimes \bfv_1)^{\top} = \bmx 4 & 2 & 2 & 1\emx,
(\bfv_2 \otimes \bfv_2)^{\top} = \bmx 1 & -1 & -1 & 1\emx,
(\bfv_3 \otimes \bfv_3)^{\top} = \bmx 1 & 2 & 2 & 4\emx.
\]
Hence, the matrix $\bfZ =\bmx \bfz_1 & \bfz_2 & \bfz_3 \emx$ is given by
\[
\bfZ^{\top} = 
\left[\begin{array}{c|cc|cccc}
-1 & -4 &  -2  &  4 &  2 & 2  & 1 \\
1 & 4 &  -4  &  1 &  -1 & -1  & 1 \\
-2 & 2 &  4  &  1 &  2 & 2  & 4 \\
\end{array}
\right]. \\
\]
\end{example}

\subsection{The tensor of Jacobian matrices of \cite{Dreesen.etal14-Decoupling}}
The tensorization method of \cite{Dreesen.etal14-Decoupling} does not use the coefficients of $\bff(\bfu)$ directly, but %is inspired on a small-signal analysis of the nonlinear function about a set of operating points. 
%The method 
proceeds by collecting the first-order information of $\bff(\bfu)$ (i.e., the partial derivatives) in a set of sampling points. 
The thusly obtained Jacobian matrices are arranged into a third-order tensor, of which the CP decomposition reveals the decomposition~(\ref{eq:decoupled}).

As in \cite{Dreesen.etal14-Decoupling}, we consider the Jacobian of $\bff$:
\begin{equation}\label{eq:jacobian_def}
\bfJ_{\bff} (\bfu) :=
\bmx
\frac{\partial f_1}{\partial u_1} (\bfu) & \cdots & \frac{\partial f_1}{\partial u_m} (\bfu) \\
\vdots & & \vdots \\
\frac{\partial f_n}{\partial u_1} (\bfu) & \cdots & \frac{\partial f_n}{\partial u_m} (\bfu)
\emx.
\end{equation}
Using Lemma~\ref{lem:jacobian_decoupled}, the tensorization is constructed as follows (see Figure~\ref{fig:tensor_DIS}):
\begin{itemize}
\item $N$ points $\bfu^{(1)}, \ldots, \bfu^{(N)} \in \bbR^m$  are chosen (so-called \emph{sampling points}).

\item An $n \times m \times N$  tensor $\TDIS$ is constructed by stacking the Jacobian evaluations at $\bfu^{(k)}$
\[
\TDIS_{:,:,k} := \bfJ_{\bff}(\bfu^{(k)}).
\]
\end{itemize}
%The construction of the tensor is visualized in Figure~\ref{fig:tensor_DIS}

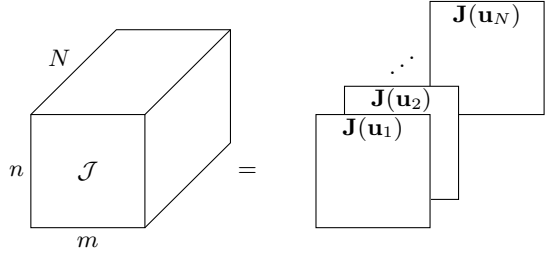
\begin{figure}[htb!]
\centering
\begin{tikzpicture}
\begin{scope}[scale=0.75]
\draw(-6,-1) rectangle (-4,1);
\draw(-6,1) -- (-4.5,2.5) -- (-2.5,2.5) -- (-2.5,0.5) --(-4,-1);
\draw(-4,1) -- (-2.5,2.5);

\draw(-2.2,0) node (D) {\small $=$};
\draw(-5,0) node (A) {\small $\TDIS$};

\draw[fill=white](1,1) rectangle (3,3);
\draw[fill=white](-0.5,-0.5) rectangle (1.5,1.5);
%\draw[fill=white](-0.75,-0.75) rectangle (1.25,1.25);
\draw[fill=white](-1,-1) rectangle (1,1);

\draw(0,0.75) node (A) {\small $\bfJ(\bfu_1)$};
\draw(0.5,1.25) node (B) {\small $\bfJ(\bfu_2)$};
\draw(2,2.75) node (C) {\small $\bfJ(\bfu_N)$};
\draw(0.5,2) node (D) {\small $\iddots$};

\draw(-5,-1.25) node (D) {\small $m$};
\draw(-6.25,0) node (D) {\small $n$};
\draw(-5.5,2) node (D) {\small $N$};

\end{scope}

\end{tikzpicture}
    \caption{The third-order tensor $\TDIS$ is constructed by stacking behind each other a set of Jacobian matrices $\bfJ$ evaluated at the sampling points $\bfu^{(k)}$. Its CP decomposition is equivalent to joint matrix diagonalization of the Jacobian matrix slices.}\label{fig:tensor_DIS}
\end{figure}

\begin{example}\label{ex:main_example_TDIS}
We continue Example~\ref{ex:main_example}. 
%The partial derivatives of $\bff(\bfu)$ are
%\begin{align*}
%\frac{\partial f_1}{\partial u_1} &= -9u_1^2-18u_1u_2-27u_2^2-16u_1-8u_2+3,\\
%\frac{\partial f_1}{\partial u_2} &= -9u_1^2-54u_1u_2-45u_2^2-8u_1-40u_2+9,\\
%\frac{\partial f_2}{\partial u_1} &= -21u_1^2-12u_1u_2+6u_2^2+20u_1+16u_2,\\
%\frac{\partial f_2}{\partial u_2} &= -6u_1^2+12u_1u_2+21u_2^2+16u_1+20u_2-3.\\
%\end{align*}
As a set of sampling points, we choose
\[
\bfu = \bmx0\\ 0\emx, \bfu^{(2)} = \bmx1\\ 0\emx, \bfu^{(3)} = \bmx0\\ 1\emx.  
\]
By evaluating $\bfJ_{\bff} (\bfu) $ at these points, we get the tensor $\TDIS$ given by
\begin{equation}\label{eq:TDIS_ex}
\TDIS_{:,:,1} = 
\bmx
3 & 9\\
0& -3
\emx,\quad 
\TDIS_{:,:,2} = 
\bmx
-22 & -8\\
-1& 7
\emx,\quad
\TDIS_{:,:,3} = 
\bmx
-32 & -76\\
22& 38
\emx.
\end{equation}
\end{example}%
If  $\bff(\bfu)$ has a decoupled representation \eqref{eq:decoupled},  the following lemma holds true.
\begin{lemma}[{\cite[Lemma 2.1]{Dreesen.etal14-Decoupling}}]\label{lem:jacobian_decoupled}
The first order derivatives of \eqref{eq:decoupled} are given by
\begin{equation}\label{eq:jacobian_decoupled}
\bfJ_{\bff} (\bfu)  = \bfW \diag(g'_1(\bfv_1^\top \bfu), \ldots, g'_r(\bfv_r^\top \bfu)) \bfV^{\top},
\end{equation}
where $g_i'(t) := \frac{dg_i}{dt}(t)$.
\end{lemma}
The proof, given in \cite{Dreesen.etal14-Decoupling}, follows by chain rule:
\[
\bfJ_{\bff} (\bfu) = \bfW \bfJ_{\bfg}(\bfV^{\top}\bfu) \bfV^{\top}. 
\]

By Lemma~\ref{lem:jacobian_decoupled},  the evaluations of the Jacobians can be jointly factorized:
\begin{equation}
\begin{split}
\bfJ(\bfu^{(1)}) &= \bfW \bfD^{(1)} \bfV^{\top}, \\
& \; \vdots \\
\bfJ(\bfu^{(N)}) &= \bfW \bfD^{(N)} \bfV^{\top},
\end{split}
\end{equation}
where $\bfD^{(k)} =  \diag(g'_1(\bfv_1^\top \bfu^{(k)}),\ldots,g'_r(\bfv_r^\top \bfu^{(k)}))$.
Therefore, $\TDIS$ admits a CP decomposition
\begin{equation}\label{eq:TDIS_CP}
\TDIS = \llbracket \bfW, \bfV, \bfH \rrbracket = \sum\limits_{k=1}^{r} \bfw_k \circ \bfv_k \circ \bfh_k,
\end{equation}
where  $\bfw_k$, $\bfv_k$ are as in \eqref{eq:decoupling_additive}, and $\bfh_{k}$ contains the evaluations of $g'_k(\bfv_k^\top \bfu)$ in $\bfu^{(1)},\ldots, \bfu^{(N)}$:
\begin{equation}\label{eq:h_k}
\bfh_k = \bmx g'_k(\bfv_k^\top \bfu^{(1)}) &\cdots & g'_k(\bfv_k^\top \bfu^{(N)})\emx^{\top}.
\end{equation}
%{The third factor matrix $\bfH = \bmx \bfh_1 & \cdots & \bfh_r \emx$ has a structure that will be revealed in the next subsections.}

\begin{example}
\label{ex:Hmatrix}
We continue Examples~\ref{ex:main_example} and~\ref{ex:main_example_TDIS}.
By differentiation, we get 
\[
\begin{split}
g'_1(t)= 3t^2 -4 t - 1, \quad
g'_2(t) = 3t^2 -8 t + 1, \quad
g'_3(t) = 3t^2 +4 t -2,
\end{split}
\]
and hence, by substitution,
\begin{equation}\label{eq:bfH_ex1}
\bfH = 
\left[ \begin{array}{rrr}
-1 & 1 & -2 \\
3  & 12 & 5\\
-2 & -4 & 18
\end{array} \right].
\end{equation}
Straightforward calculations show indeed that $\TDIS$ given in \eqref{eq:TDIS_ex} admits a decomposition \eqref{eq:TDIS_CP} with $\bfH$ as in \eqref{eq:bfH_ex1}.
\end{example}

\section{Relation between tensorizations $\TDIS$ and $\TVVU$}
\label{sec:relations}
In this section, we show how CP decompositions of (\ref{eq:TVVU_CP}) and (\ref{eq:TDIS_CP}) are related. 
Moreover, we establish the relation between the ranks of the tensors and uniqueness of CP decompositions.

{\color{blue}
First, we show the relation between the vectors $\bfz_k$ and $\bfh_k$, defined in \eqref{eq:z_k} and \eqref{eq:h_k}, respectively.
}
We give the proof of this basic fact for completeness.
\begin{lemma}\label{lem:zk_and_hk}
{\color{blue}
The vectors $\bfz_k$ and $\bfh_k$ defined in \eqref{eq:z_k} and \eqref{eq:h_k}, respectively, satisfy 
}
\begin{equation}\label{eq:TDIS_TVVU_CP_EQ}
\bfh_k = \VDM^{\top} \bfz_k,
\end{equation}
where $\VDM\in \bbR^{\delta \times N}$ is a Vandermonde-like matrix whose columns are 
\begin{equation}\label{eq:VDM}
\VDM_{:,j} = 
\left[\begin{array}{c|c|c|c|c} 1 & 2 (\bfu^{(j)})^{\top} & 3 (\bfu^{(j)}\otimes \bfu^{(j)})^{\top} & \cdots & d (\bfu^{(j)}\otimes \cdots \otimes \bfu^{(j)})^{\top} \end{array}\right]^{\top}.
\end{equation}
\end{lemma}
{
\begin{proof}
Recall that by the properties of the Kronecker product
\[
(\underbrace{\bfu \otimes \cdots \otimes \bfu}_{d \text{ times }})^{\top} (\bfv \otimes \cdots \otimes \bfv) = (\bfu^{\top} \bfv)^d.
\]
Then from \eqref{eq:z_k} have that 
\[
(\VDM^{\top} \bfz_k)_j = \VDM_{:,j}^{\top}  \bfz_k = 
c_{k,1} + c_{k,2} (\bfv_{k}^{\top}\bfu^{(j)}) + \cdots + c_{k,d} (\bfv_{k}^{\top}\bfu^{(j)})^{d-1} = (\bfh_k)_j,
\]
where the last equality follows from \eqref{eq:h_k} and the fact that
\[
g'_k(t) = c_{k,1} + c_{k,2} t + \cdots + c_{k,d} t^{d-1}.
\]
\end{proof}}

\begin{example}
In Example~\ref{ex:main_example_TDIS}, the matrix $\VDM$ can be found as
\[
\VDM^{\top} =  
\left[\begin{array}{c|cc|cccc}
1 & 0 & 0  &  0  &  0 & 0 & 0 \\
1 & 2 & 0  &  3 &  0  &  0  & 0 \\
1 & 0 & 2  & 0 &  0 &  0 & 3 \\
\end{array}
\right]. 
\]
It is easy to see that $\bfH = \VDM^{\top} \bfZ$.
\end{example}

As a consequence, we get that the two tensors and their ranks are also related.
\begin{theorem}\label{prop:TDIS_TVVU_EQ}
\begin{enumerate}
\item For any polynomial map $\bff$, $\TDIS$ and $\TVVU$ are related as  
\begin{equation}\label{eq:TDIS_TVVU_EQ}
\TDIS=  \TVVU \bullet_3 \VDM^{\top}.
\end{equation}

\item The rank of $\VDM$  is bounded as
\[
\rrank \VDM \le M := \binom{m+d-1}{d-1}.
\]
In addition, if $M \le N$, and $M$ points in $\{\bfu^{(j)}\}$ are  in general position, then $\rrank \VDM = M$.
For example, if points $\{\bfu^{(j)}\}$ are independent and sampled from a continuous probability distribution, then  $\rrank \VDM = M$ with probability $1$.

\item If $\VDM$ has maximal possible rank (i.e. $\rrank \VDM = M$), then
\[
\rrank \TDIS = \rrank \TVVU,
\]
and all the minimal CP decompositions differ only by the third factors, which are linked as in \eqref{eq:TDIS_TVVU_CP_EQ}.
Moreover, if the CP decomposition of $\TVVU$ is unique, then the CP decomposition of $\TDIS$ is also unique.
\end{enumerate}
\end{theorem}

\begin{proof}[Proof of Lemma~\ref{lem:zk_and_hk}]
Let us express $g'_k(\bfv_k^\top \bfu)$ in an explicit form.
First, $g'_{k}(t) = c_{k,1} + 2 c_{k,2} t + 3 c_{k,3} t^2 + \cdots +  d c_{k,d} t^{d-1}$, from which it follows that
\[
g'_{k}(\bfv_k^\top \bfu) = c_{k,1} + 2 c_{k,2}  \bfv_k^\top \bfu + 3 c_{k,3}  ( \bfv_k^\top \bfu )^2 + \cdots +d c_{k,d}  ( \bfv_k^\top \bfu)^{d-1}.
\]
Since $( \bfv^{\top} \bfu )^{s} = (\bfv \otimes \cdots \otimes \bfv)^{\top} (\bfu \otimes \cdots \otimes \bfu)$,  the $j$-th element of $\bfh_{k}$ is equal to
\[
    (\bfh_{k})_j = h_{j,k} = g'_{k}(\bfv_k^\top \bfu^{(j)})  = \VDM_{:,k}^{\top} \bfz_k,
\]
which completes the proof.
\end{proof}

\begin{proof}[Proof of Theorem~\ref{prop:TDIS_TVVU_EQ}] 
1. First, any polynomial map $\bff$ can be decomposed as \eqref{eq:decoupled} with $r$ sufficiently large.
Let us take such a decomposition; then it holds that
\[
(\TVVU) \bullet_3 \VDM^{\top} = \left(\sum\limits_{k=1}^{r} \bfw_k \circ \bfv_k \circ  \bfz_k\right) \bullet_3 \VDM^{\top}  = \sum\limits_{k=1}^{r} \bfw_k \circ \bfv_k \circ  \VDM^{\top} \bfz_k = \TDIS,
\]
where the last equality follows from \eqref{eq:TDIS_TVVU_CP_EQ}.

\noindent 2. By construction, each element in the image of $\VDM$ lies in the following subspace:
\begin{equation}\label{eq:struct_vdm_rows}
\begin{split}
\scrA := \{ & \bmx a_0 &  \bfa_1^{\top} & \bfa_2^{\top} & \cdots & \bfa_{d-1}^{\top}\emx^{\top} \in \bbR^{\delta} | \\
& \bfa_k \in \bbR^{n^k} \mbox{ is a vectorization of a symmetric }m \times \cdots \times m \mbox{ tensor}. \}
\end{split}
\end{equation}
Taking into account that the dimension of the space of $m \times \cdots \times m$ symmetric tensors of order $s$ is $\binom{m+s-1}{s}$,
we get that the maximal possible  rank of $\VDM$ is 
\[
1+m+\binom{m+1}{2}+\dots \binom{m+d-2}{d-1} =M.
\]
Next, from \eqref{eq:VDM}, we have that the $k$-th column contains evaluations of all $M$ monomials $\{u^{j_1}_1 \cdots u^{j_m}_m \le d-1 \}^{j_1+\cdots+j_m \le d}_{j_1,\ldots,j_m=0}$ at a point $\bfu^{(k)}$ (scaled by a constant).
If, without loss of generality, the first $M$ points $\{\bfu^{(k)}\}_{k=1}^{M}$ are in general position, then the columns of $\bfA$ corresponding to different monomials are linearly independent by \cite[Multiplicity One Theorem]{Miranda1999}, hence $\rrank \bfA = M$.

\noindent 3.
Note that each tube $\TVVU_{i,j,:}$ of the tensor $\TVVU$, by construction, lies in $\scrA$.
If $\rrank \bfA = M$, then its row span coincides with $\scrA$. 
Hence the following identity holds true:
\begin{equation}\label{eq:TVVU_TDIS_EQ}
\TDIS \bullet_3 (\VDM^{\dagger})^{\top}=  \TVVU \bullet_3 (\VDM \VDM^{\dagger})^{\top} = \TVVU. 
\end{equation}
The remaining properties follow from \eqref{eq:TVVU_TDIS_EQ} and \eqref{eq:TDIS_TVVU_EQ}.
\end{proof}

\section{Structured tensor decompositions}
\label{sec:structured}
\subsection{From CPD to a decomposition with structured rank-one terms} 
The CP decomposition of $\TDIS$ and $\TVVU$, although related, are not always equivalent to the original decomposition \eqref{eq:decoupling_additive}.
This happens because there are still nontrivial linear dependencies between  the elements of $\TVVU$ and $\TDIS$. 
In what follows, we establish relationships between the CP decompositions and the original decomposition \eqref{eq:decoupling_additive}.

First, we prove that for the rank-one case, these decompositions coincide.
\begin{proposition}\label{prop:rank_one}
Consider a polynomial map $\bff(\bfu)$ of degree $d$, and 
the tensor $\TVVU$ built from it.
Then the following holds
\[
\rrank (\TVVU) \le 1 \iff \bff(\bfu) = \bfw g(\bfv^{\top}\bfu),
\]
where $\bfw \in \bbR^{n}$, $\bfv \in \bbR^{m}$ and $g(t)$ is a polynomial of degree $d$.
\end{proposition}
\begin{proof}
The $\boxed{\Leftarrow}$ follows from \Cref{lem:TVVU_decoupled_cp}.
Let us prove the $\boxed{\Rightarrow}$ part. Assume that %$\TVVU$ has  the decomposition
\[
\TVVU = \bfw \circ \bfv \circ \bfy.
\]
First, since the tensor $\TVVU$ contains all the coefficients of the derivatives, we have that there exists a polynomial $\widetilde{f}(\bfu)$ such that $\nabla f_k (\bfu) = (\bfw)_k \nabla \widetilde{f}(\bfu)$.
Since the polynomials $f_k(\bfu)$ do not have constant terms, we have that
\[
\bff(\bfu) = \bfw \widetilde{f}(\bfu),
\]
where $\Psi(\widetilde{f}) = \bfv  \bfy^{\top}$.

{Next, let us show that the  polynomial $\widetilde{f}$ should necessarily the form $\widetilde{f}(\bfu) = {g}(\bfv^{\top} \bfu)$. Since $\Psi(\widetilde{f}) = \bfv  \bfy^{\top}$, then it follows from \eqref{eq:def_Psi} that all the unfoldings $\HomT{1}$,  $\HomT{2}$, $\HomT{3}_{(1)}, \cdots, \HomT{d}_{(1)}$ have rank at most one and their column space is spanned by the vector $\bfv$.
Therefore, we have that
\[
\begin{split}
& \HomT{1} = c_1 \bfv, \\
& \HomT{2} = c_2 \bfv \bfv^{\top}, \\
& \HomT{3} = c_3 \bfv \circ \bfv \circ \bfv, \\
& \vdots \\
& \HomT{d} = c_d \bfv \circ \cdots \circ \bfv, \\
\end{split}
\]
and hence $\widetilde{f}(\bfu) = {g}(\bfv^{\top} \bfu)$ where 
\[
{g}(t) = {c}_1 t + {c}_2 t^2 +  \cdots + {c}_d t^d,
\]
which completes the proof.}
\end{proof}

{\begin{remark}
The fact that $\rrank \Psi(\widetilde{f}) \le  1$ implies  $\widetilde{f}(\bfu) = g(\bfv^{\top} \bfu)$ also can be proved alternatively, by noting that the matrix $\Psi{\widetilde{f}}$, after removing duplicate columns, can be reduced to the form $S(f)$ in \cite[Proposition 22]{comon2017xrank}.
Hence, by \cite[Proposition 4.1]{comon2017xrank}, the polynomial $\widetilde{f}$ has necessarily the form $\widetilde{f}(\bfu) = g(\bfv^{\top} \bfu)$.
However, this alternative proof requires introducing extra notation, which would be much longer that the proof presented in this paper.
\end{remark}}

\begin{corollary}
    If the $N$ sampling points are chosen such that the  $\rrank (\VDM) = M$, then
\[
\rrank (\TDIS) \le 1 \iff \bff(\bfu) = \bfw g(\bfv^{\top}\bfu).
\]
\end{corollary}

As a corollary of Proposition~\ref{prop:rank_one}, we get that the original polynomial decomposition~(\ref{eq:decoupling_additive}) is equivalent to a structured CP decomposition.
\begin{corollary}
Let $\LVVU \subset \bbR^{n\times m\times \delta}$ be the linear subspace of tensors with the structure of $\TVVU$. 
Let the sampling points be chosen such that $\rrank \VDM = M$, and $\LDIS \subset \bbR^{n\times m\times N}$ be the linear subspace of tensors with the structure of $\TDIS$. 

Then the following three statements are equivalent:
\begin{enumerate}
\item the polynomial map $\bff(\bfu)$ admits a decomposition \eqref{eq:decoupling_additive};
\item the tensor $\TVVU(\bff)$ admits the structured CP decomposition
\begin{equation}\label{eq:structured_CPD_TVVU}
    \TVVU = \TVVU_1 + \cdots + \TVVU_r, \quad \rrank (\TVVU_k) = 1,\quad  \TVVU_k \in  \LVVU;
\end{equation}
\item the tensor $\TDIS(\bff)$ admits the structured CP decomposition
\begin{equation}\label{eq:structured_CPD_TDIS}
    \TDIS = \TDIS_1 + \cdots + \TDIS_r, \quad \rrank (\TDIS_k) = 1,\quad  \TDIS_k \in  \LDIS.
\end{equation}
\end{enumerate}
\end{corollary}

The structure constraint is important: indeed, the CP decomposition of the tensor $\TVVU$ or $\TDIS$ is not necessarily structured.
In general, we do not know even if the CP rank is equal to the structured CP rank (minimal number of terms in (\ref{eq:structured_CPD_TVVU}) or~(\ref{eq:structured_CPD_TDIS})).
This is similar to the Comon's conjecture \cite[\S 5]{Comon.etal08SJMAA-Symmetric} about symmetric tensors: it is not known whether the symmetric rank of a symmetric tensor equals its non-symmetric rank. 

However, if the CP decomposition of a tensor is unique (for example, if it satisfies Kruskal's uniqueness conditions), then it should necessarily be a structured CP decomposition.

\subsection{Computing coupled/structured CP decomposition}\label{sec:simcoupsymCPD}
Earlier attempts to tackle the structured case were made by~\cite{Tiels.Schoukens14conf-coupled,Schoukens.etal14conf-Identification} and \cite[\S 8, pp.~133--136]{GabrielPhD}. 
The attempts of~\cite{Tiels.Schoukens14conf-coupled,Schoukens.etal14conf-Identification} have the disadvantage that a tensor is built that has missing values, which increase in number as the polynomial degree grows. 
The attempt of~\cite{GabrielPhD} consisted of parameterizing the internal { nonlinear functions $g_k$} using their coefficients. 
Although this seems a promising approach, it turned out to be problematic in practice to build a working algorithm, as the {  decoupling} method led to strongly nonlinear/nonconvex optimization problems.

We propose to tackle the problem by solving a coupled and structured CP decomposition instead. 
First, let us consider simultaneous decoupling of homogeneous polynomials \eqref{eq:decoupling_sim_HomT}.
Let us arrange the $\HomT{d}_i$, for $i = 1,\ldots,n$ into a tensor $\calT^d$, such that $\calT^d_{i,:,\ldots,:} = \HomT{d}_i$, for all $i = 1,\ldots,n$.
Then it is easy to verify that $\calT^d$ admits a partially symmetric CP decomposition 
\[
\calT^d = \llbracket \bfW, \underbrace{\bfV, \ldots, \bfV}_{\mbox{$d$ times}} \rrbracket,
\]
which, in our decoupled representation (\ref{eq:decoupled}), takes the form
$ \bfW \bfg (\bfV^\top \bfu)$,
where $\bfg(\bfx) = \left[ \begin{array}{ccc} x_1^d & \cdots & x_r^d \end{array} \right]^\top$.%, and $\bfW$ serves to make a linear combination of all the $r$ internal branches for each of the $n$ function components.  

Decoupling non-homogeneous polynomials can be achieved by means of a coupled structured CP decomposition of the $\calT^d$ tensors.
Let us arrange all $\HomT{d}_i$, for $i = 1,\ldots,n$ into the tensors $\calT^d$ (like in the previous paragraph), such that $\calT^d_{i,:,\ldots,:} = \HomT{d}_i$, for all $i = 1,\ldots,n$. 
We now have for each degree a coupled partially symmetric CP decomposition as  
\begin{equation}
    \begin{array}{rcl}
        \calT^1 &=& \llbracket \bfW, \bfV, \bfc_1^\top \rrbracket,\\
        \calT^2 &=& \llbracket \bfW, \bfV, \bfV, \bfc_2^\top \rrbracket,\\
        &\vdots&  \\
        \calT^d &=& \llbracket \bfW, \underbrace{\bfV, \ldots, \bfV}_{\mbox{$d$ times}}, \bfc_d^\top \rrbracket,
    \end{array}
    \label{eq:coupledsymCPD}
\end{equation}
where the $\bfc_i$, for $i = 1,\ldots,d$, are the $i$-th degree coefficients for each of the $r$ {nonlinear functions $g_k$}.

Remark that these coefficients were not required in the previous paragraphs when homogeneous polynomials were considered: in such cases  { the nonlinear functions $g_k$ are of the form $c_k t^d$, i.e., they differ only by a scaling factor}, which can be assumed to be fully absorbed by $\bfW$. 
Also remark that there { are redundancies} in the representation \eqref{eq:coupledsymCPD}: { for example}, an equivalent problem { can be} obtained { if one rescales a  coefficient vector $\bfc_\delta$ to a} vector containing ones, in which case a rescaling has to take place on the remaining coefficients as well as on $\bfW$. 
Finally we want to mention that the framework of structured data fusion \cite{sorber2015sdf,tensorlab3} allows for computing tensor decompositions as in \eqref{eq:coupledsymCPD}, where several tensors (and possibly matrices) are jointly decomposed while sharing factors, possibly while imposing structure on the factors.

{
\begin{example}
Let us continue with Examples~\ref{ex:main_example},~\ref{ex:main_example_TDIS} and~\ref{ex:Hmatrix}. 
We have that $m=n=2$ and $r=3$, which does not guarantee a unique CP decomposition of $\TDIS$ (under assumptions of genericity, see \cite{Dreesen.etal14-Decoupling}). 
Indeed, if we compute a numerical CP decomposition of tensor $\TDIS$, we find that, up to a relative norm-wise error $2.3546 \times 10^{-16}$, $\TDIS$ admits a CP decomposition with factors 
\[ 
\begin{array}{rcl}
\tilde{\bfW} &=& 
\left[ 
\begin{array}{rrr}
    1.1628 &  -3.2951 &   3.0252 \\
    0.5705 &   1.1349 &  -2.1791 \\
\end{array}
\right],
\quad
\tilde{\bfV} =
\left[ 
\begin{array}{rrr}
    3.5822 & -0.7705 &  -2.2959\\
   -0.0226 & -3.4455 &  -2.9785\\
\end{array}
\right], 
\\ \\
\tilde{\bfH} &=&
\left[ 
\begin{array}{rrr}
    0.2736 &   0.8181 &   0.0312\\
   -3.3900 &   0.2647 &   1.2313\\
    0.5334 &  -3.9194 &   3.4945\\
\end{array}
\right], \\
\end{array}
\]
the columns of which are not scaled and permuted versions of the columns of $\bfW$, $\bfV$, $\bfH$. 

It can be shown that the `structured CP' approaches are able to correctly return the underlying factors $\bfW$, $\bfV$ and $\bfH$ (up to scaling and permutation invariances). 
For instance, the structured data fusion framework \cite{sorber2015sdf,tensorlab3} is able to compute the coupled and partially symmetric decomposition (\ref{eq:coupledsymCPD}).
This returns 
\[ 
\begin{array}{rcl}
\tilde{\bfW} &=&
\left[ \begin{array}{rrr} 
    1.2767  &  1.7112 &   0\\
         0  & -0.8556 &  -1.9980
\end{array} \right], \quad
\tilde{\bfV} = 
\left[ \begin{array}{rrr} 
   -0.9218 &  -1.0534  &  1.5879 \\
    0.9218 &  -2.1067  &  0.7940 
\end{array} \right], 
\end{array}
\]
    as well as computed values for the coefficient vectors of $g_i(x_i)$, which are omitted here.
It can be verified that $\tilde\bfW$ and $\tilde\bfV$ are scaled and permuted versions of  $\bfW$ and $\bfV$.

Remark that if one uses $m=n=r=2$, both the structured and non-structured CP decomposition return the same decomposition (up to scaling and permutation of the columns of the factors). 
Indeed, in this case, uniqueness is guaranteed (generically), ensuring that the underlying factors are identifiable.
    This could be checked easily by generating a variation of the equations that we are decoupling where the third columns of $\bfV$ and $\bfW$ are removed, so that $g_3(x_3)$ is not considered.  
\end{example}

{
\subsection{Linking $\TVVU$ and $\calT^1,\ldots,\calT^{d}$ tensors}
In this section, we show how $\TVVU$  and its CPD is connected with the tensors $\calT^s$ and their joint decomposition \eqref{eq:coupledsymCPD}.
Let $(1,2)$-reshapings of the tensors $\calT^s$ to be the third order tensors $\calT^s_{(1,2)} \in \bbR^{n\times m \times m^{s-1}}$ defined as
\[
(\calT^s_{(1,2)})_{i,j,:} = \mvec(\calT^s_{i,j,:,\cdots,:}).
\]
Then it is easy to see that the tensor $\calT^s_{(1,2)}$  can be split into slices as shown in Fig.~\ref{fig:Ts_slices}, where the $\HomT{s,j}$ is the symmetric tensor corresponding to the $s$-th degree homogeneous part of the polynomial $f_k$.

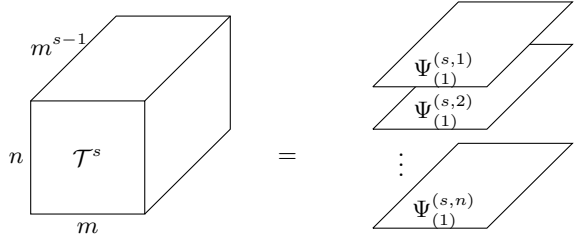
\begin{figure}[htb!]
\centering
\begin{tikzpicture}
\begin{scope}[scale=0.75]
\draw(-6,-1) rectangle (-4,1);
\draw(-6,1) -- (-4.5,2.5) -- (-2.5,2.5) -- (-2.5,0.5) --(-4,-1);
\draw(-4,1) -- (-2.5,2.5);

\draw(-1.5,0) node (D) {\small $=$};
\draw(-5,0) node (A) {\small $\calT^{s}$};

\draw(0.5,0) node (D) {\small $\vdots$};

\draw[fill=white](0,-1.25) -- (1.5,0.25) -- (3.5,0.25) -- (2,-1.25) -- (0,-1.25);
\draw[fill=white](0,0.5) -- (1.5,2) -- (3.5,2) -- (2,0.5) -- (0,0.5);
\draw[fill=white](0,1.25) -- (1.5,2.75) -- (3.5,2.75) -- (2,1.25) -- (0,1.25);

\draw(1.25,1.55) node (B) {\small $\HomT{s,1}_{(1)}$};
\draw(1.25,0.8) node (B) {\small $\HomT{s,2}_{(1)}$};
\draw(1.25,-0.95) node (B) {\small $\HomT{s,n}_{(1)}$};

\draw(-5,-1.25) node (D) {\small $m$};
\draw(-6.25,0) node (D) {\small $n$};
\draw(-5.5,2) node (D) {\small $m^{s-1}$};

\end{scope}
\end{tikzpicture}
    \caption{The slices of the tensor $\calT^{s}$ are the unfoldings of the symmetric tensors corresponding to $s$-degree homogeneous parts of polynomials $f_1,\ldots,f_n$.}\label{fig:Ts_slices}
\end{figure}

By taking into account the definition \eqref{eq:def_Psi} of the slices of the tensor $\TVVU$, we can easily see that $\TVVU$ can be constructed by stacking the tensors is equivalent to reshaping the tensors $\calT^s_{(1,2)}$ along the third mode together, as shown in Fig.~\ref{fig:TVVU_Ts}.

\begin{figure}[htb!]
\centering
\begin{tikzpicture}
\begin{scope}[scale=0.75]
\draw(-6,-1) rectangle (-4,1);
\draw(-6,1) -- (-3,4) -- (-1,4) -- (-1,2) --(-4,-1);
\draw(-4,1) -- (-1,4);

\draw(-1.5,0) node (D) {\small $=$};
\draw(-5,0) node (A) {\small $\TVVU$};

\draw(0,-1) rectangle (2,1);
\draw(0,1) -- (3,4) -- (5,4) -- (5,2) --(2,-1);
\draw(2,1) -- (5,4);

\draw(1.2,2.2) -- (3.2,2.2) -- (3.2,0.2) ;
\draw(0.2,1.2) -- (2.2,1.2) -- (2.2,-0.8) ;
\draw(1.8,2.8) -- (3.8,2.8) -- (3.8,0.8) ;

\draw(1,0) node (A) {\small $\calT^{1}_{(1,2)}$};
\draw(2.7,0.6) node (A) {\small $\calT^{2}_{(1,2)}$};
\draw(3.45,1.7) node (A) {\small $\iddots$};
\draw(4.5,2.4) node (A) {\small $\calT^{d}_{(1,2)}$};

\draw(-5,-1.25) node (D) {\small $m$};
\draw(-6.25,0) node (D) {\small $n$};
\draw(-4.5,3) node (D) {\small $\delta$};

\end{scope}
\end{tikzpicture}
    \caption{Stacking the reshapings of $\calT^{s}$ together.}\label{fig:TVVU_Ts}
\end{figure}
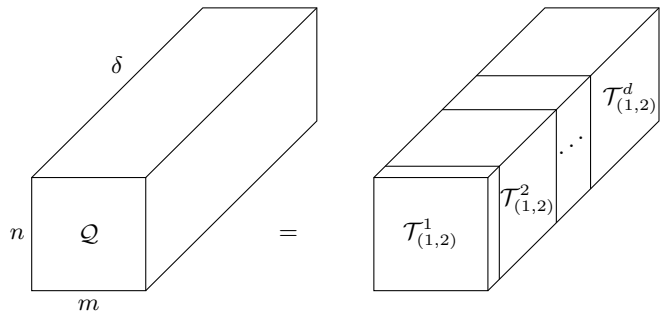

Remark that in \Cref{lem:VVU_CPD} we see that the structure appearing in the CP decomposition of $\TVVU$ is closely connected to the simultaneous decomposition described in \Cref{sec:simcoupsymCPD}.
Indeed, \Cref{lem:VVU_CPD} can be alternatively deduced from \eqref{eq:coupledsymCPD}, because the outer products of vectors become Kronecker products after reshaping. 
 }

%TODO: Explain that we can view $\TVVU$ as a concatenation of $\calT^d$ tensors (skype figure of Kostya on whiteboard). 

\section{Conclusions and {perspectives}}\label{sec:conclusions}
% Findings
We { have} established a link between two tensorization approaches for decoupling multivariate polynomials \cite{VanMulders.etal14conf-Identification,Dreesen.etal14-Decoupling}:
the tensor of Jacobian matrices \cite{Dreesen.etal14-Decoupling} can be obtained by multiplying the coefficient-based tensor \cite{VanMulders.etal14conf-Identification} by a Vandermonde-like matrix.
As revealed by this connection, the two approaches have similar fundamental properties, such as equal tensor rank { and uniqueness of the CP decomposition under conditions on the number and location of the sampling points. }

{
The decoupling problem, however, is not equivalent to the CP decomposition of one of the tensors. 
This may lead to  loss of uniqueness and identifiability of the CP decomposition, in the cases when the original decomposition is still unique.
We have shown that by adding structure to the CP decomposition we can obtain equivalence between tensor decomposition and decoupling problems for polynomials.
The structure can be imposed either as a joint decomposition of partially symmetric tensors, or can be imposed on rank-one factors.
Numerical experiments confirm that using structured decompositions can restore uniqueness of the polynomial decoupling.
 }

%It is well-known that tensors often have a unique CP decomposition, which is a sufficient condition for guaranteeing a uniquely identifiable decoupled representation. 
%We extended this uniqueness property for the decoupling problem further: 
%By considering a joint CP decomposition, where the different tensor decompositions share factors, the solution space was reduced, and we observed that a larger class of decoupling problems was identifiable.  

% Conclusions
%The obtained { results} shed a light 
Although our { results show that different tensor-based approaches are very closely related, let us make some remarks on} applicability of the approaches { and some future directions}. 
%on the applicability of the two approaches. %First, properties can now be transferred from one approach to the other. For example, uniqueness aspects in one of the settings would lead to uniqueness aspects in the other. Second, as the underlying principles of the approaches are different, each of them is suitable for solving different variations and generalizations of the main decoupling problem. Thus, understanding the advantages and disadvantages of the approaches helps choosing the more suitable one. 
%We mention that, 
%for dealing with higher-order polynomials, the approach based on Jacobian matrices would be more appropriate, as the size of the involved tensor is independent of the order of the polynomials. 
For (differentiable) non-polynomial functions, the approach based on Jacobian matrices would be more appropriate, as it only uses evaluations of the derivatives of the functions. 
{ Coefficient-based approach seems  more relevant in the case when the region of interest is unclear, or } when some of the coefficients are missing or unreliable.
{ In both cases, an interesting open question remains how to impose the structure directly on the rank-one components, without resorting to coupled tensor factorizations.
Another important question is how to address the approximate decoupling problem, i.e., when we are dealing with noise (see \cite{hollander2018} for results on the unstructured case).}

\section{Acknowledgements}\noindent
This work was supported in part by the ERC AdG-2013-320594 grant ``DECODA'',
by the Flemish Government (Methusalem),
by the ERC Advanced Grant SNLSID under contract 320378, and 
by the Fonds Wetenschappelijk Onderzoek -- Vlaanderen under EOS Project no 30468160 and FWO projects G.0280.15N and G.0901.17N.

\section*{References}
\bibliographystyle{elsarticle-num}
\bibliography{tensors3}
\end{document}